\documentclass[a4paper,10pt]{article}

\usepackage{amsmath, amssymb, amsthm} % Für Symbole, Formeln und Theoreme

\DeclareMathOperator{\chr}{char}
\DeclareMathOperator{\sign}{sign}

\usepackage{fancybox, framed}
\usepackage{braket} % für Erstellen von Mengen, deren Mengenklammern und Bedingungsstrich perfekt angepasst werden

\usepackage[urlbordercolor={.7 .5 .1}]{hyperref}
\usepackage{pdfpages}
\usepackage{mathtools} % Für das Symbol ":="

\usepackage{cite} % Für Zitate in Bibliography

\usepackage{graphicx} % Um Bilder einzufügen

\usepackage{tabularx} % fuer Tabellen
\usepackage{multirow} % um Zeilen in Tabelle zu verbinden

\usepackage{enumerate} % Für römische Aufzählung

\usepackage{pgf,tikz}  
\usepackage{subcaption}
\usepackage{caption}
\usepackage{mathrsfs}
\usetikzlibrary{arrows} % Um Bilder aus Geogebra einzufügen

\usepackage{gensymb} % Für das Symbol degree

\usepackage{cleveref} % Für Referenzen von Sätzen, benutze \Cref oder \cref (für gross und Kleinschreibung)

\theoremstyle{definition}
\newtheorem{definition}{Definition}[section]
\newtheorem{example}[definition]{Example}
\newtheorem{fact}[definition]{Fact}
\theoremstyle{corollary}
\newtheorem{corollary}[definition]{Corollary}
\newtheorem{lemma}[definition]{Lemma} 
\newtheorem{proposition}[definition]{Proposition}
\newtheorem{theorem}[definition]{Theorem}
 
\theoremstyle{definition} 
 
\usepackage{fancyhdr}
\newcommand{\N}{\boldsymbol{\mathrm{N}}}
\title{Maximal Circular Point Sets over Arbitrary Fields and an Application to Cryptography}
\date{2024, November}
\author{Chris Busenhart \\ Departement of Mathematics, ETH Zürich}
\begin{document}

\maketitle

\begin{abstract}
\noindent % um Einrückung zu verhindern
The study of rational point sets on circles over the Euclidean plane is discussed in a more general framework, i.e. we generalize the notion “rational" and consider these circular point sets over arbitrary fields. We also determine the cardinality of maximal circular point sets which depends on the radius of the corresponding circle and the characteristic of the underlying field. For the construction of them we use the so called perfect distances which have the necessary compatibility properties to find new points on a circle such that all these points still have rational distance from each other. Then we define the rotation group where its elements are the points on a circle over an arbitrary field and find a connection between a subgroup of it and perfect distances if our field is a prime field. Furthermore, we describe a possible application in cryptography of the rotation group similar to the Diffie-Hellman key exchange.
\end{abstract}

\begin{section}{Introduction and framework}

% eventuell möglich, ganzes Bild auf einmal anzupassen
%{\fontsize{Fontgröße}{Grundlinienabstand} \selectfont}

	%In my master thesis I studied rational circular point sets on the Euclidean plane. 
%Strongly related with rational point sets are integral point sets which have been studied for more than thousand years \cite{HeikHarIntDist}. Also
%The study of integral point sets is more than thousand years old \cite{HeikHarIntDist}. 
The motivation to consider point sets in $ \mathbb{R}^n $ such that Euclidean distances between any pair of points from this set are rational came from the study of the so called integral point sets, where all the mutual distances of its points are integral, e.g. see \cite{Solymosi1, SaschaKurz, Kreisel_Kurz_int_hept, Kurz_Wassermann_min_dia}. %i.e. sets of points in $ \mathbb{R}^n $ such that all distances between any pair of two points of this set are integral. 
These integral point sets have been investigated for several thousand years \cite{HeikHarIntDist}. One of the most famous results in this field is the Erd\H{o}s-Anning theorem \cite{ErdAnnIntDist, Erdoes2} which tells us that every integral point set of infinite cardinality in the plane is contained in a straight line. Additionally, Anning found $ 12 $ points on a circle such that all mutual distances are integral and the lengths of these distances are solutions of a Diophantine equation \cite{NorAnn}. He also conjectured that it is possible to find such integral circular point sets with cardinality $ 24 $ and $ 48 $. Later Friedelmeyer explained how the $ 12 $ points of Anning \cite{FriedMey} can be constructed and in $ 2020 $, Halbeisen and Hungerbühler \cite{HalbHung} generalized the algorithm of Anning to find $ 3 \cdot 2^n $ such points on a common circle for $ n \in \mathbb{N} $. Another procedure to construct integral circular point sets of arbitrarily finite cardinality is described in \cite{my_masterthesis}. There is also recent work about integral and rational point sets \cite{Solymosi2}.

On the other hand, circular point sets of infinite cardinality clearly exists in the Euclidean plane and it is always possible to scale the circle of every finite rational circular point set such that one obtains an integral circular point set. Therefore the construction of rational circular point sets is strongly related to the construction of integral point sets and the previous ones were already described by Euler \cite{Euler}. The aim of this chapter is to generalize the procedure for the construction of circular point sets in affine planes over arbitrary fields. Instead of working with the Euclidean distance we use the so called quadrance introduced in \cite{Wildberger} which goes back to Wildberger and which can also be used analogously over finite field planes (compare also with \cite{ LeAnhVinh1, LeAnhVinh2, LeAnhVinh3}). Later Kurz called it (squared) distance, see \cite{Kier_Kurz1, Sascha_Kurz2, Kier_Kurz2}.

However, our focus is on maximal point sets on circles over finite field planes where we require all squared distances between any two points of the considered set to be a square in the underlying prime field. We see later why this is a natural extension of rational circular point sets in the Euclidean plane for arbitrary field planes.

For the construction of maximal circular point sets we use the so called perfect distances which we introduce later. For this we have to explain first what we mean with the words “circle", “distance" and “rational" if we work with affine planes over other fields than the real numbers what we will do in the next section after recalling and introducing some basic notion which will be used in the sequel part.
%So let us introduce the following basic definitions and examples first.
\end{section}

\begin{section}{Basic notions and tools}

\begin{subsection}{Definitions and facts}

We use the following conventions: For an arbitrary field $ \mathbb{F} $, we denote the set of the multiplicative group of $ \mathbb{F} $ by $ \mathbb{F}^{*} = \mathbb{F} \setminus \left\{ 0 \right\} $. Moreover, we denote the smallest integer $ n \in \mathbb{N} $ such that the sum of $ n $ copies of the multiplicative neutral element is equal to the the additive neutral element by $ \chr \left( \mathbb{F} \right) $ and call it characteristic. If such an $ n $ does not exist, then we set $ \chr \left( \mathbb{F} \right) = 0 $. Moreover, a field without a proper subfield is called prime field and we denote the prime field of $ \mathbb{F} $ by $ \boldsymbol{P} \left( \mathbb{F} \right) $. We use the symbol $ \square_{\mathbb{F}} \coloneqq \left\{ a^2 \mid a \in \mathbb{F} \right\}	$ for the set of (multiplicative) squares in $ \mathbb{F} $.
	
In the following we will recall some of the well-know statements about finite fields which we will use later.

	\begin{fact} \label{fact1_finite_fields}
	For any finite field $ \mathbb{F} $ there exists a prime number $ p \in \mathbb{N} $ and an integer $ n \in \mathbb{N} \setminus \left\{ 0 \right\} $ such that $ \vert \mathbb{F} \vert = p^n $ and, conversely, for each prime power $ p^n $ there is a finite field $ \mathbb{F} $ such that $ \vert \mathbb{F} \vert = p^n $. Moreover, the field $ \mathbb{F} $ with cardinality $ p^n $ is unique up to isomorphism. 
	\end{fact}
	
As a small abuse of notation we write that two fields are equal if and only if they are isomorphic. We denote a finite field of cardinality $ p^n $ by $ \mathbb{F}_{p^n} $.

\begin{fact}	\label{fact2_finite_fields}
	Let $ n,m \in \mathbb{N} \setminus \left\{ 0 \right\} $, then $ \mathbb{F}_{p^n} $ is a subfield of $ \mathbb{F}_{p^m} $ if and only if $ n \mid m $.
\end{fact}
	
In particular, for each finite prime field $ \mathbb{F}_{p^n} $ we have $ \boldsymbol{P} \left( \mathbb{F} \right) = \mathbb{F}_p $. Moreover, if $ \mathbb{F} $ is a field with $ \chr \left( \mathbb{F} \right) = 0 $, then $ \boldsymbol{P} \left( \mathbb{F} \right) = \mathbb{Q} $, where $ \mathbb{Q} $ denotes the field of the rational numbers \cite[p.\,7-8]{McCarthy}.
	
We denote an arbitrary square root of $ -1 $ by $ \sqrt{-1} $ and we write $ \sqrt{-1} \in \mathbb{F} $ to mention that $ \mathbb{F} $ contains a square root of $ -1 $. Whereas $ \sqrt{-1} \notin \mathbb{F} $ means the opposite.
	
\begin{fact} \label{fact3_finite_fields}
	Let $ \mathbb{F} $ be a finite field. Then 
	$$ \sqrt{-1} \notin \mathbb{F} \Longleftrightarrow\vert \mathbb{F} \vert \equiv 3 \pmod{4} .$$
\end{fact}	
	
Now we would like to generalize maps like translations and rotations as we know it from the Euclidean plane for arbitrary field planes. Points in the field planes are denoted by row and column vectors of two entries depending on what is more convenient to work with.

%where we separate the entries by a comma. Where its transpose will be a column vector without comma. To switch between row and column vector we use the transpose sign $ ^T $.

\begin{definition}
	Let $ \mathbb{F} $ be a field and $ P = \left(p_1, p_2 \right), Q = \left(q_1, q_2 \right) \in \mathbb{F}^2 $ be an arbitrary point. Then we can define the {\em translation $ \tau_{Q}: \mathbb{F}^2 \to \mathbb{F}^2 $ in direction $ Q $} by
	$$ \tau_{Q} \left( P \right) =
		 \begin{pmatrix}
			p_1 + q_1 \\
			p_2 + q_2
		\end{pmatrix}  .$$
Let $ a,b \in \mathbb{F} $ with $ a^2 + b^2 = 1 $, then the {\em rotation around the origin with parameters $ a,b $} $ \theta^{a,b}: \mathbb{F}^2 \to \mathbb{F}^2 $ is defined by 
	$$ \theta^{a,b} \left( P \right) = 
		\begin{pmatrix}
			a & b \\
			-b & a
		\end{pmatrix} 
		 \begin{pmatrix}
			p_1 \\
			p_2
		\end{pmatrix} $$
where the operation in between is matrix multiplication. Moreover, a {\em rotation around} $ M \in \mathbb{F}^2 $ {\em with parameters $ a,b $} is a composition of translations and rotations around the origin in the following way: $ \theta_{M}^{a,b} : \mathbb{F}^2 \to \mathbb{F}^2 $ is defined by $ \theta_{M}^{a,b} = \tau_{M} \circ \theta^{a,b} \circ \tau_{M}^{-1} $. We call
$$ T^\mathbb{F} \coloneqq \Set{ \tau_{P} | P \in \mathbb{F}^2 } $$	
{\em set of translations} and for $ M \in \mathbb{F}^2 $ we call
	$$ \Theta_{M}^{\mathbb{F}} \coloneqq \Set{\theta_M^{a,b} | a^2 + b^2 = 1} $$
{\em set of rotations around any point} $ M \in \mathbb{F}^2 $.
\end{definition}

Observe that the above sets equipped with the composition defines each an abelian group that acts on the points of an arbitrary field plane.% transitively.

\begin{definition}
	Let $ \mathbb{F} $ be an arbitrary field, $ M = (m_1,m_2) \in \mathbb{F}^2 $ and $ r \in {\mathbb{F}}^{*} $. Then we call 
	\begin{displaymath}
		C \left(M,r \right)_{\mathbb{F}} \coloneqq \Set{ \left( x,y \right) \in \mathbb{F}^2 | \left( x - m_1 \right)^2 + \left( y - m_2 \right)^2 = r^2 }
	\end{displaymath}
the {\em circle given by its  center $ M $ and radius $ r $}.
\end{definition}

Note that we do not want to allow $ r = 0 $. In case $ \sqrt{-1} \notin \mathbb{F} $ a circle with radius $ r = 0 $ would only contain its center. This might happen if $ \mathbb{F} $ is a finite field and $ \vert \mathbb{F} \vert \equiv 3 \pmod{4} $. Whereas a finite field $ \mathbb{F} $ with $ \vert \mathbb{F} \vert \equiv 1 \pmod{4} $ has $ \sqrt{-1} \in \mathbb{F} $, i.e. there exist $ a,b \in \mathbb{F}^{*} $ such that $ a^2 + b^2 = 0 $. Then a circle with radius $ r = 0 $ can consist of two lines, see \Cref{ex_f5_zero_radius}.

% and not both are equal to zero is equivalent to the existence of a square root of $ -1 $ in $ \mathbb{F} $. I.e. for finite fields with $ \vert \mathbb{F} \vert \equiv 3 \pmod{4} $ we would get that a circle defined by a vanishing radius would contain just the center. In case $ \vert \mathbb{F} \vert \equiv 1 \pmod{4} $ there are still more points contained in $ C \left(M,r \right)_{\mathbb{F}} $. However, they might be collinear what we will see in \Cref{ex_f5_zero_radius}.

%The following easy example will lead us through the next few definitions and statements.

\begin{example} \label{CF_7}
	%We have that
	\begin{displaymath}
		C \left((0,0),1 \right)_{\mathbb{F}_7} = \left\{ \left( 0,1 \right), \left( 0,6 \right),\left( 1,0 \right),\left( 2,2 \right),\left( 2,5 \right),\left( 5,2 \right),\left( 5,5 \right),\left( 6,0 \right) \right\}
	\end{displaymath}
is a circle over the affine plane $ \mathbb{F}_7 \times \mathbb{F}_7 $ consisting of exactly eight different points.
\end{example}

% In the last example it is easy to find the points in $ C \left(1,(0,0) \right) $ by hand since we can just check each of the $ 49 $ points in $ \mathbb{F}_7 \times \mathbb{F}_7 $ whether they satisfy the equation of the circle or not. Later we will give a parametrisation for all these points such that it will be easy to find them directly.
The points in \Cref{CF_7} can be determined by plugging in all elements of $ \mathbb{F}_7 \times \mathbb{F}_7 $ or with the help of a parametrization which we will see later. Next we generalize the notion of the Euclidean distance for points on planes over arbitrary fields.
%The next step will be a generalisation of the squared Euclidean distance we have over the real plane.

\begin{definition}
	Let $ P = \left( p_1, p_2 \right) $ and $ Q = \left( q_1, q_2 \right) $ be two points on an affine plane over the arbitrary field $ \mathbb{F} $. Then we call
	\begin{displaymath}
		D^2 \left( P,Q \right) \coloneqq \left( p_1 - q_1 \right)^2 + \left( p_2 - q_2 \right)^2
	\end{displaymath}
{\em squared distance between $ P $ and $ Q $}. 
\end{definition}

Even though square roots can be defined over all fields \cite{square_root} we avoid working with them for the sake of simplicity.
%Although there exist square root functions on every field, see \cite{square_root}, we avoid to work with them for the sake of simplicity.
Observe that $ D $ is generally far away from being a metric although symmetry is satisfied. In fact, we do not have a partial order on arbitrary fields like “$ \leq $” or “$ \geq $” on $ \mathbb{R} $ and the squared distance of two points in an affine plane could vanish even if these points are not identical. However, we will see that the squared distances between different points on a circle over an affine plane either all vanish or none of them do, see \Cref{dist_lemma}. 

From observations in the real plane we know that translations and rotations around any point applied to two point in the Euclidean plane does not change the Euclidean distance. We will show now that the squared distance remains invariant when we apply elements from $ T^\mathbb{F} $ and $ \Theta_{M}^{\mathbb{F}} $, too.

\begin{lemma} \label{squared_distance_invariant}
	Let $ \mathbb{F} $ be an arbitrary field and $ P,Q \in \mathbb{F}^2 $. Then the squared distance of $ P,Q $ remains invariant under translations and rotations.
\end{lemma}

\begin{proof}
	Let $ P = \left( p_1,p_2 \right), Q= \left( q_1,q_2 \right), R = \left( r_1,r_2 \right) \in \mathbb{F}^2 $ and $ a,b \in \mathbb{F} $ such that $ a^2+b^2 = 1 $ be arbitrary, then we have
	\begin{align*}
		D^2 \left( \tau_R \left( P \right),\tau_R \left( Q \right) \right) &= D^2 \left( 
		\begin{pmatrix} 
		p_1+r_1 \\
		p_2+r_2 \end{pmatrix},
		\begin{pmatrix} 
		q_1+r_1 \\
		q_2+r_2 \end{pmatrix} 
		\right) \\
		&= \left( \left( p_1+r_1 \right)-\left( q_1+r_1 \right) \right)^2 + \left( \left( p_2+r_2 \right)-\left( q_2+r_2\right) \right)^2 \\
		&= D^2 \left( P,Q \right).
	\end{align*}
Now we show that a rotation around the origin does not change the value of the squared distance, i.e.
	\begin{align*}
		D^2 \left( \theta^{a,b} \left( P \right),\theta^{a,b} \left( Q \right) \right)
		&= D^2 \left( \begin{pmatrix}
			a & b \\
			-b & a
		\end{pmatrix} 
		\left( \begin{matrix}
		p_1 \\
		p_2
\end{matrix}	\right)	 , 
 \begin{pmatrix}
			a & b \\
			-b & a
		\end{pmatrix} 
		\left( \begin{matrix}
		q_1 \\
		q_2
\end{matrix}	\right) \right)  \\
		&= D^2 \left( 
		\begin{pmatrix} 
		ap_1+bp_2 \\
		-bp_1+ap_2 
		\end{pmatrix}, 
		\begin{pmatrix} 
		aq_1+bq_2 \\
		-bq_1+aq_2 
		\end{pmatrix} \right)\\[-14pt]
		\intertext{ \center{$ = \left( a \left( p_1-q_1\right) + b \left(p_2-q_2 \right)  \right)^2 + \left( -b \left( p_1-q_1 \right) + a \left( p_2-q_2 \right) \right)^2 $}}\\[-20pt]
		&= \left( a^2 + b^2 \right) \left( \left( p_1-q_1\right)^2 + \left( p_2-q_2\right)^2 \right) \\
		&= D^2 \left( P,Q \right).
	\end{align*}		
Observe that $ \theta^{a,b} =  \theta_{\left( 0,0 \right)}^{a,b} $. Since $ \theta_{M}^{a,b} = \tau_{M} \circ \theta^{a,b} \circ \tau_{M}^{-1} $ is a decomposition of translations and a rotation around the origin, i.e. a decomposition of elements which lets the squared distance invariant, we conclude that $ \theta_{M}^{a,b} $ does so, too.
\end{proof}

Now we would like to generalize the notion “rational". After that we are able to decide whether a distance between two points on a circle is rational or not such that it remains compatible with our understanding of a rational distance in the case $ \mathbb{F} = \mathbb{R} $. We know that $ \mathbb{Q} $ is the prime field of $ \mathbb{R} $, so the squared distance between two points with rational Euclidean distance is just a square in the underlying prime field of the real numbers. This motivates the following definition:

% $ \boldsymbol{\mathrm{P}} \hspace{-0.5mm} \left( \mathbb{F} \right) $ $ \pf $
\begin{definition}
	Let $ \mathbb{F} $ be an arbitrary field and %and $ \boldsymbol{P} \left( \mathbb{F} \right) $ denote its prime field. 
let $ P $ and $ Q $ be points of the affine plane $ \mathbb{F} \times \mathbb{F} $. If 
	$$ D^2 \left( P,Q \right) \in \square_{\boldsymbol{P} \left( \mathbb{F} \right)} ,$$ i.e. $ D^2 \left( P,Q \right) $ has a root in $ \boldsymbol{P} \left( \mathbb{F} \right) $, then we say that the squared distance of $ P $ and $ Q $ is {\em rational} or in short form: $ P $ and $ Q $ have {\em r.s.d.}.
\end{definition}

\begin{example} \label{ex_F7}
	We would like to find the rational squared distances of the points in $ C \left((0,0),1 \right)_{\mathbb{F}_7} $. Observe that 
	\begin{displaymath}
		\square_{\mathbb{F}_{7}} \coloneqq \left\{ 0,1,2,4 \right\}.
	\end{displaymath}	 
Hence, we can calculate all the $ 28 $ distances between the points in $ C \left((0,0),1 \right)_{\mathbb{F}_7} $ and check whether they are rational or not. In \Cref{F_7_with_distances} you see the rational distances coloured blue and the non-rational distances marked in red. 
% where the numbering of the points is in the same order as they appear in the set of $ C \left(1,(0,0) \right)_{\mathbb{F}_7} $ in \Cref{CF_7}
\end{example}

%%%%%%%%%%%%%%%%%%%%%%%%%%%%%%%%%%%%%%%%%%%%%%%%%	
\begin{figure}[h]  
\begin{center}
\pagestyle{empty} 
\definecolor{ffqqqq}{rgb}{1.,0.,0.}
\definecolor{qqqqff}{rgb}{0.,0.,1.}
\definecolor{sqsqsq}{rgb}{0.12549019607843137,0.12549019607843137,0.12549019607843137}
\begin{tikzpicture}[line cap=round,line join=round,>=triangle 45,x=0.375cm,y=0.375cm]
\clip(-13.,-12.) rectangle (13.,12.);
\draw [line width=2.pt,color=qqqqff] (3.8268343236508984,-9.238795325112866)-- (9.238795325112868,-3.8268343236508953);
\draw [line width=2.pt,color=qqqqff] (-3.8268343236508966,-9.238795325112868)-- (-9.238795325112866,-3.8268343236508966);
\draw [line width=2.pt,color=qqqqff] (-9.238795325112866,3.8268343236508997)-- (-3.8268343236508966,9.238795325112868);
\draw [line width=2.pt,color=qqqqff] (3.8268343236508944,9.238795325112868)-- (9.238795325112866,3.8268343236509006);
\draw [line width=2.pt,color=qqqqff] (-9.238795325112866,3.8268343236508997)-- (9.238795325112866,3.8268343236509006);
\draw [line width=2.pt,color=qqqqff] (9.238795325112866,3.8268343236509006)-- (-3.8268343236508966,9.238795325112868);
\draw [line width=2.pt,color=qqqqff] (-9.238795325112866,3.8268343236508997)-- (3.8268343236508944,9.238795325112868);
\draw [line width=2.pt,color=qqqqff] (3.8268343236508944,9.238795325112868)-- (-3.8268343236508966,9.238795325112868);
\draw [line width=2.pt,color=qqqqff] (-9.238795325112866,-3.8268343236508966)-- (9.238795325112868,-3.8268343236508953);
\draw [line width=2.pt,color=qqqqff] (3.8268343236508984,-9.238795325112866)-- (-9.238795325112866,-3.8268343236508966);
\draw [line width=2.pt,color=qqqqff] (-3.8268343236508966,-9.238795325112868)-- (9.238795325112868,-3.8268343236508953);
\draw [line width=2.pt,color=qqqqff] (3.8268343236508984,-9.238795325112866)-- (-3.8268343236508966,-9.238795325112868);
\draw [line width=2.pt,color=ffqqqq] (9.238795325112868,-3.8268343236508953)-- (9.238795325112866,3.8268343236509006);
\draw [line width=2.pt,color=ffqqqq] (9.238795325112868,-3.8268343236508953)-- (3.8268343236508944,9.238795325112868);
\draw [line width=2.pt,color=ffqqqq] (9.238795325112868,-3.8268343236508953)-- (-3.8268343236508966,9.238795325112868);
\draw [line width=2.pt,color=ffqqqq] (9.238795325112868,-3.8268343236508953)-- (-9.238795325112866,3.8268343236508997);
\draw [line width=2.pt,color=ffqqqq] (3.8268343236508984,-9.238795325112866)-- (9.238795325112866,3.8268343236509006);
\draw [line width=2.pt,color=ffqqqq] (3.8268343236508984,-9.238795325112866)-- (3.8268343236508944,9.238795325112868);
\draw [line width=2.pt,color=ffqqqq] (3.8268343236508984,-9.238795325112866)-- (-3.8268343236508966,9.238795325112868);
\draw [line width=2.pt,color=ffqqqq] (3.8268343236508984,-9.238795325112866)-- (-9.238795325112866,3.8268343236508997);
\draw [line width=2.pt,color=ffqqqq] (-3.8268343236508966,-9.238795325112868)-- (9.238795325112866,3.8268343236509006);
\draw [line width=2.pt,color=ffqqqq] (-3.8268343236508966,-9.238795325112868)-- (3.8268343236508944,9.238795325112868);
\draw [line width=2.pt,color=ffqqqq] (-3.8268343236508966,-9.238795325112868)-- (-3.8268343236508966,9.238795325112868);
\draw [line width=2.pt,color=ffqqqq] (-3.8268343236508966,-9.238795325112868)-- (-9.238795325112866,3.8268343236508997);
\draw [line width=2.pt,color=ffqqqq] (-9.238795325112866,-3.8268343236508966)-- (9.238795325112866,3.8268343236509006);
\draw [line width=2.pt,color=ffqqqq] (-9.238795325112866,-3.8268343236508966)-- (3.8268343236508944,9.238795325112868);
\draw [line width=2.pt,color=ffqqqq] (-9.238795325112866,-3.8268343236508966)-- (-3.8268343236508966,9.238795325112868);
\draw [line width=2.pt,color=ffqqqq] (-9.238795325112866,-3.8268343236508966)-- (-9.238795325112866,3.8268343236508997);
\begin{scriptsize}
\draw [fill=sqsqsq] (-3.8268343236508966,9.238795325112868) circle (2.5pt);
\draw[color=sqsqsq] (-4.6,10.2) node {\fontsize{10}{0} $ \left( 5,2 \right) $};
\draw [fill=sqsqsq] (9.238795325112866,3.8268343236509006) circle (2.5pt);
\draw[color=sqsqsq] (10.55,4.3) node {\fontsize{10}{0} $ \left( 2,2 \right) $};
\draw [fill=sqsqsq] (-9.238795325112866,-3.8268343236508966) circle (2.5pt);
\draw[color=sqsqsq] (-10.8,-4.3) node {\fontsize{10}{0} $ \left( 6,0 \right) $};
\draw [fill=sqsqsq] (3.8268343236508984,-9.238795325112866) circle (2.5pt);
\draw[color=sqsqsq] (4.4,-10.35) node {\fontsize{10}{0} $ \left( 0,6 \right) $};
\draw [fill=sqsqsq] (9.238795325112868,-3.8268343236508953) circle (2.5pt);
\draw[color=sqsqsq] (10.55,-4.3) node {\fontsize{10}{0} $ \left( 1,0 \right) $};
\draw [fill=sqsqsq] (3.8268343236508944,9.238795325112868) circle (2.5pt);
\draw[color=sqsqsq] (4.4,10.2) node {\fontsize{10}{0} $ \left( 2,5 \right) $};
\draw [fill=sqsqsq] (-9.238795325112866,3.8268343236508997) circle (2.5pt);
\draw[color=sqsqsq] (-10.8,4.3) node {\fontsize{10}{0} $ \left( 5,5 \right) $};
\draw [fill=sqsqsq] (-3.8268343236508966,-9.238795325112868) circle (2.5pt);
\draw[color=sqsqsq] (-4.6,-10.35) node {\fontsize{10}{0} $ \left( 0,1 \right) $};
\draw[color=qqqqff] (7.,-7.) node {\fontsize{10}{0} $ 2 $};
\draw[color=qqqqff] (-7.15,-7.) node {\fontsize{10}{0} $ 2 $};
\draw[color=qqqqff] (-7.2,7.) node {\fontsize{10}{0} $ 2 $};
\draw[color=qqqqff] (6.9,7.) node {\fontsize{10}{0} $ 2 $};
\draw[color=qqqqff] (-0.1,3.1) node {\fontsize{10}{0} $ 4 $};
\draw[color=qqqqff] (1.7,6.1) node {\fontsize{10}{0} $ 2 $};
\draw[color=qqqqff] (-2.,6.1) node {\fontsize{10}{0} $ 2 $};

\draw[color=qqqqff] (-0.1,10.1) node {\fontsize{10}{0} $ 4 $};
\draw[color=qqqqff] (-0.1,-3.15) node {\fontsize{10}{0} $ 4 $};
\draw[color=qqqqff] (-2,-6.2) node {\fontsize{10}{0} $ 2 $};
\draw[color=qqqqff] (1.7,-6.2) node {\fontsize{10}{0} $ 2 $};
\draw[color=qqqqff] (-0.1,-10.) node {\fontsize{10}{0} $ 4 $};
\end{scriptsize}
\end{tikzpicture}

\caption{The points of $ C \left((0,0),1 \right)_{\mathbb{F}_7} $ with rational squared distances (blue) and non-rational squared distances (red)} 
\label{F_7_with_distances}
\end{center}
\end{figure}
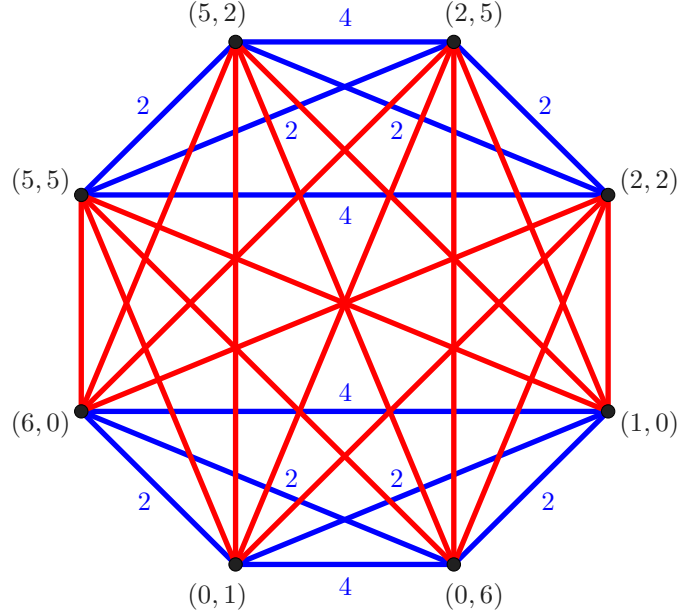
%%%%%%%%%%%%%%%%%%%%%%%%%%%%%%%%%%%%%%%%%%%%%%%%%	

Recall that so far a rational point set in the real plane is a set of points such that their mutual Euclidean distances are all rational (observe that we do not require the coordinates of the points to be rational). We will now generalize this notion for arbitrary field planes by the help of squared distances.

\begin{definition}
	Let $ \mathbb{F} $ be any field. Then we call a set $ S \subset \mathbb{F}^2 $ {\em rational point set} if and only if for all $ P,Q \in S $ we have that $ P $ and $ Q $ have r.s.d..
\end{definition}	

In \Cref{F_7_with_distances} we see that the points of $ C \left((0,0),1 \right)_{\mathbb{F}_7} $ can be divided into two groups of points such that all points of one group have rational squared distances from each other. 
In the next section we would like to examine these special subsets on circles over prime field planes. Then we will consider circles on planes over arbitrary fields. But before we start with this, let us introduce the following definition.

\begin{definition}
	Let $ \mathbb{F} $ be an arbitrary field, $ M = (m_1,m_2) \in \mathbb{F}^2 $ and $ r \in {\mathbb{F}}^{*} $. Then we call a rational point set $ S \subset \mathbb{F}^2$ {\em circular point set} if $ S \subseteq C \left(M,r \right)_{\mathbb{F}} $. %,i.e. all points in $ S $ have rational squared distance from each other. 
Moreover, we call $ S $ {\em e-maximal with respect to $ C \left(M,r \right)_{\mathbb{F}} $} if we cannot extend $ S $ with any further point of $ C \left(M,r \right)_{\mathbb{F}} $ such that the resulting set remains a circular point set. Additionally, we call an e-maximal subsets of $ C \left(M,r \right)_{\mathbb{F}} $ {\em c-maximal} if they have maximal cardinality among all e-maximal subsets of $ C \left(M,r \right)_{\mathbb{F}} $.
\end{definition}

%Since circular point sets are translation invariant, it suffices to consider the case $ M = \left( 0,0 \right) $.
Observe, that the letter “e" of e-maximal stands for “extension", i.e. it is not possible to extend this rational circular point set such that the extension remains rational and similarly, “c" stands for “circle" in the sense of that we find no larger rational point set on a given circle. Since squared distances are invariant with respect to translation, we will often choose $ M = \left( 0,0 \right) $ as center of a circular point set for simplicity. We will see later an example which shows the existence of e-maximal circular point sets which are not c-maximal, see \Cref{ex_F7_extended}.
\end{subsection}

\begin{subsection}{Parametrization of points on circles and consequences}

%At first we would like to describe the points on a centred circle by a parametrisation because this will have a lot of advantages for proofs we will do later.

Let $ \mathbb{F} $ be any field and $ r \in \mathbb{F}^{*} $. We would like to find all points on $ C \left((0,0),r \right)_{\mathbb{F}} $ which can be described by the algebraic equation
	\begin{align*}
		x^2 + y^2 = r^2.
	\end{align*}
At first we assume that $ \chr \left( \mathbb{F} \right) \neq 2 $. Clearly $ \left( 0,-r \right) \in C \left((0,0),r \right)_{\mathbb{F}} $. Now for $ \left( 0,-r \right) $ and any other point of $  C \left((0,0),r \right)_{\mathbb{F}} $, there is a unique line in the affine plane $ \mathbb{F} \times \mathbb{F} $ containing the two points. Hence, we can describe all points on $ C \left((0,0),r \right)_{\mathbb{F}} $ as intersection points of lines through $ \left( 0,-r \right) $ with the circle given by the above equation. Observe that every secant containing $ \left( 0, -r \right) $ is either of the form
	\begin{align*}
		y = tx - r \ \ \mathrm{or} \ \ x = 0
	\end{align*}
for $ t \in \mathbb{F} $. % Therefore we only have to calculate the second intersection point of the above lines with the circle. 
Obviously, the equation $ x = 0 $ gives us the point $ \left( 0,r \right) $. In all other cases, we can insert the linear equation in the equation of the circle and we get
	\begin{align*}
		\left( t^2 + 1 \right)x^2 - 2trx + r^2 = r^2.		
	\end{align*} 
If we assume $ t^2 \neq -1 $ and $ x \neq 0 $, the last equation gives us the following solution for $ x $:
	\begin{align*}
		x = \dfrac{2tr}{t^2 + 1}
	\end{align*} 
Plugging $ x $ in the equation of the secant, we finally get
	\begin{align*}
		y = 	\dfrac{r \left( t^2 - 1 \right)}{t^2 + 1}.
	\end{align*}
%which gives us
%	\begin{align*}
%		C \left(r,(0,0) \right)_{\mathbb{F}} = \Set{ \left( \frac{2tr}{t^2 + 1}, 	\frac{r \left( t^2 - 1 \right)}{t^2 + 1} \right) | t \in \mathbb{F}, t^2 \neq -1} \cup \Set{ \left( r,0 \right) }.
%	\end{align*}
Then for every admissible $ t $ we have another solution which is different from $ \left( 0,\pm r \right) $. 

On the other hand, in case $ \chr \left( \mathbb{F} \right) = 2 $, it is easier to find the points because then every element of $ \mathbb{F} $ has exactly one root. Hence, the equation of the circle can be solved for $ y $:
	\begin{displaymath}
		y = x + r
	\end{displaymath}	
To get a parametrization of a general circle with center $ (m_1, m_2) $, we can just shift all the points by these coordinates. Finally, we conclude the next statement.

% Thus, we can conclude (compare also with Sascha Kurz...):

%Eventuell hier limes von t gegen unendlich definieren!

\begin{corollary} \label{coro_parametrisation}
	Let $ \mathbb{F} $ be an arbitrary field, $ r \in {\mathbb{F}}^{*} $ and $ \left( m_1,m_2 \right) \in \mathbb{F} \times \mathbb{F} $. 
If $ \chr \left( \mathbb{F} \right) = 2 $, then we have
	\begin{align*}
		C \left((m_1,m_2),r \right)_{\mathbb{F}} = \Set{ \left( m_1 + t, m_2 + t + r \right) | t \in \mathbb{F}}.
	\end{align*}	
If $ \chr \left( \mathbb{F} \right) \neq 2 $, we have
	
	\begin{gather*}
		C \left((m_1,m_2),r \right)_{\mathbb{F}} = \Set{ \left( m_1 + \frac{2tr}{t^2 + 1}, m_2 + \frac{r \left( t^2 - 1 \right)}{t^2 + 1} \right) | t \in \mathbb{F}, \ t^2 \neq -1} \\ \cup \Set{ \left( m_1 , m_2 + r \right) }.
	\end{gather*}	

Furthermore, the cardinality of the set above is

\begin{align*}
	\rvert C \left((m_1,m_2),r \right)_{\mathbb{F}} \rvert = \left\{
\begin{array}{lll}
\lvert \mathbb{F} \rvert &  \textrm{if} \ \chr \left( \mathbb{F} \right) = 2 \\ [1.12ex]
\lvert \mathbb{F} \rvert -1 & \textrm{if} \ \chr \left( \mathbb{F} \right) \neq 2 \ \textrm{and} \ \sqrt{-1} \in \mathbb{F} \\ [1.12ex]
\lvert \mathbb{F} \rvert + 1 &  \textrm{otherwise} \\
\end{array}
\right. 
\end{align*}	
	
where it could be that $ \lvert \mathbb{F} \rvert = \infty $. In this case we set $ \lvert \mathbb{F} \rvert \pm 1 = \infty $.	
	
\end{corollary}

The next goal is to describe the squared distance between points in $ C \left((0,0),r \right)_{\mathbb{F}} $ in terms of parameters. For this we will prove the following lemma.

\begin{lemma} \label{dist_lemma}
	Let $ C \left( (m_1,m_2),r \right) $ be a circle in an arbitrary affine plane $ \mathbb{F} \times \mathbb{F} $ with $ \chr \left( \mathbb{F} \right) \neq 2 $ and
%Let $ t_1,t_2 \in \mathbb{F} $ such that $ t_1^2 + 1 \neq 0 \neq t_2^2 + 1 $ and 
let $ P_1,P_2 $ be the points parametrized by $ t_1,t_2 \in \mathbb{F} $ such that $ t_j^2 \neq -1 $ for $ j = 1,2 $ as in \Cref{coro_parametrisation}, respectively. 
	Then we have
	\begin{align*}
		D^2 \left( P_1, P_2 \right) = \dfrac{4r^2 \left( t_1 - t_2 \right)^2}{ \left( t_1^2 + 1 \right) \left( t_2^2 + 1\right)}
	\end{align*}
and
	\begin{align*}
		D^2 \left( P_1, \left( m_1 , m_2 + r \right) \right) = \dfrac{4r^2}{ \left( t_1^2 + 1 \right)}.
	\end{align*}		
Furthermore, if $ \chr \left( \mathbb{F} \right) = 2 $, then all distances between points of $ C \left(  \left(m_1,m_2 \right),r \right) $ vanish.
	
\end{lemma}

\begin{proof}
	At first let $ \chr \left( \mathbb{F} \right) \neq 2 $. By calculation we get
		\begin{align*}
			D^2 \left( P_1,P_2 \right) &= \left( \frac{2t_1r}{t_1^2 + 1} - \frac{2t_2r}{t_2^2 + 1} \right)^2 + \left( \frac{r \left( t_1^2 - 1 \right)}{t_1^2 + 1} - \frac{r \left( t_2^2 - 1 \right)}{t_2^2 + 1} \right)^2 \\
	&= \frac{r^2}{\left( t_1^2 + 1 \right)^2 \left( t_2^2 + 1 \right)^2} \left[ 4 \left( t_1 t_2^2+ t_1 - t_2 t_1^2 -t_2  \right)^2 + 4\left( t_1^2-t_2^2 \right)^2 \right] \\
	&= \frac{4r^2\left( t_1-t_2\right)^2}{\left( t_1^2 + 1 \right)^2 \left( t_2^2 + 1 \right)^2} \left[ \left( 1- t_1t_2 \right)^2 + \left( t_1 + t_2 \right)^2 \right] \\
	&= \frac{4r^2\left( t_1-t_2\right)^2}{\left( t_1^2 + 1 \right)^2 \left( t_2^2 + 1 \right)^2} \left[ t_1^2t_2^2 + t_1^2 + t_2^2 + 1 \right] \\
	&= \frac{4r^2 \left( t_1-t_2\right)^2}{\left( t_1^2 + 1 \right) \left( t_2^2 + 1 \right)}
	\end{align*}
and	
	\begin{align*}
			D^2 \left( P_1, \left( m_1 , m_2 + r \right) \right) &= \left( \frac{2t_1r}{t_1^2 + 1} \right)^2 + \left( \frac{r \left( t_1^2 - 1 \right)}{t_1^2 + 1} - r  \right)^2 \\
		&= \frac{r^2}{\left( t_1^2 + 1 \right)^2} \left[4t_1^2 + \left( t_1^2-1-t_1^2-1 \right)^2 \right] \\
	&= \frac{4r^2}{t_1^2+1}.		
	\end{align*}
	
If $ \chr \left( \mathbb{F} \right) = 2 $, then	
	$$ D^2 \left( P_1,P_2 \right) = \left( t_1-t_2 \right)^2 + \left( t_1 - t_2 \right)^2 = 0 $$
and
	$$ D^2 \left( P_1, \left( m_1 , m_2 + r \right) \right) = t_1^2 + t_1^2 = 0 .$$	
\end{proof}

%The proof of the lemma can be done just by calculation and it is straight forward. % so will not be showed. 
%The last point of \Cref{dist_lemma} might be a bit unexpected if we compare this with the following. Indeed, by using \Cref{dist_lemma} we directly conclude the following.
By using \Cref{dist_lemma} we can deduce the next statement.
	\begin{corollary} \label{coro_no_vanish}
		If $ \mathbb{F} $ is an arbitrary field with $ \chr \left( \mathbb{F} \right) \neq 2 $, then there are no vanishing squared distances between different points of a circle in $ \mathbb{F} \times \mathbb{F} $.
	\end{corollary}
	
In case $ \chr \left( \mathbb{F} \right) = 2 $ one can also show  that each squared distance between any two points of a finite Cartesian product of commutative rings with $ 1 $ is a square with respect to the underlying ring, see \cite[Lemma 2]{Sascha_Kurz2}.

Note that any affine plane over an arbitrary field $ \mathbb{F} $ contains two different points with squared distance equal to zero if and only if the element $ -1 $ has a root in $ \mathbb{F} $. However, on a circle over an affine plane, either all mutual squared distances between the points are zero or none of them vanishes.

\begin{example} \label{ex_f5_zero_radius}
	Consider the finite field $ \mathbb{F}_5 $. Since $ 2^2 = 4 = -1 $, the affine plane $ \mathbb{F}_5 $ contains different points with squared distance equal to zero. For example, the squared distance of the points $ P_1 \coloneqq \left( 1,2 \right) $ and $  P_2 \coloneqq \left( 2,4 \right) $ vanishes. By \Cref{coro_no_vanish} this means that there exists no circle in $ \mathbb{F}_5 \times \mathbb{F}_5 $ containing $ P_1 $ and $ P_2 $. If we allow the radius of such a circle containing both of these points to be equal to zero, then the center of this circle would be collinear to $ P_1 $ and $ P_2 $. This is because the center $ M = \left( m_1,m_2 \right) \in \mathbb{F}_5^2 $ of a circle containing both points must satisfy the equations:
	\begin{align*}
		\left( 1 - m_1 \right)^2 + \left( 2 - m_2 \right)^2 = 0 \\
		\left( 2 - m_1 \right)^2 + \left( 4 - m_2 \right)^2 = 0		
	\end{align*}		
If we simplify both equations, then we get $ 2m_1 = m_2 $ and the points satisfying them are all collinear to  $ P_1 $ and $ P_2 $. In particular, $ \left( 0,0 \right) $ is such a point and if we choose $ M = \left( 0,0 \right) $, then the equation $ x^2 + y^2 = 0 \in \mathbb{F}_5[x,y] $ is not irreducible, i.e. it splits in its components $ y + 2x = 0 $ and $ y + 3x = 0 $. The set of solutions such that at least one of the equations is satisfied is $ 9 $ here, so we have more solutions than $ 4 $ which \Cref{coro_parametrisation} would suggest. By the same statement we get that a splitting of the equality $ x^2 + y^2 = r $ over $ \mathbb{F}_5 $ can only happen if and only if $ r = 0 $. These observations justify why we exclude $ r = 0 $ in the equation of the circle. Whereas other conics as the parabola $ y = x^2 $ in the same affine plane clearly contains $ P_1 $ and $ P_2 $. 
\end{example}

\end{subsection}
\end{section}

\begin{section}{Maximal circular point sets}

\begin{subsection}{The case if $ \mathbb{F} $ is a prime field}

In this section we would like to identify the e-maximal and c-maximal circular point sets on circles in affine prime field planes. %and so $ \mathbb{F} $ will denote a prime field for the  whole section. 
But before we start with this, let us consider a further example.
% By \Cref{ex_F7} we could think that all prime fields contain exactly two disjoint maximal circular point sets such that their union is the whole circle. In fact, this is true if the characteristic of the considered prime field is a finite odd number.

\begin{example}
	In \Cref{F_7_with_distances} we considered $ \mathbb{F}_7 $, a prime field which does not contain any square root of $ -1 $. Hence, let us now examine the circle 
	\begin{displaymath}
		C \left((7,11),6 \right)_{\mathbb{F}_{13}} \subseteq \mathbb{F}_{13} \times \mathbb{F}_{13}.
	\end{displaymath}	
Since $ 5^2 = 25 = -1 $, the field $ \mathbb{F}_{13} $ contains a square root of $ -1 $. Thus, the circle contains twelve elements
	\begin{gather*}
		C \left((7,11),6 \right)_{\mathbb{F}_{13}} = \left\{ \left( 7,5 \right), \left( 0,11 \right),\left( 4,12 \right),\left( 8,8 \right),\left( 6,1 \right),\left( 10,10 \right), \right. \\ \left. \left( 4,10 \right),\left( 8,1 \right), \left( 6,8 \right), \left( 10,12 \right), \left( 1,11 \right),\left( 7,4 \right) \right\}
	\end{gather*}
by \Cref{coro_parametrisation} and the set of all rational squared distances is 
	\begin{displaymath}
		\square_{\mathbb{F}_{13}} \coloneqq \left\{ 0,1,3,4,9,10,12 \right\}.
	\end{displaymath}
Now we can determine the rational squared distances between the points as one can see in \Cref{F_13_with_distances}. The violet and dark blue lines are rational squared distances and the points without a connecting line have non-rational squared distance from each other. We clearly see that the e- and c-maximal circular point sets of $ C \left((7,11),6 \right)_{\mathbb{F}_{13}} $ are the two disjoint subsets connected either by violet or dark blue lines. 
\end{example}

%%%%%%%%%%%%%%%%%%%%%%%%%%%%%%%%%%%%%%%%%%%%%%%%%	
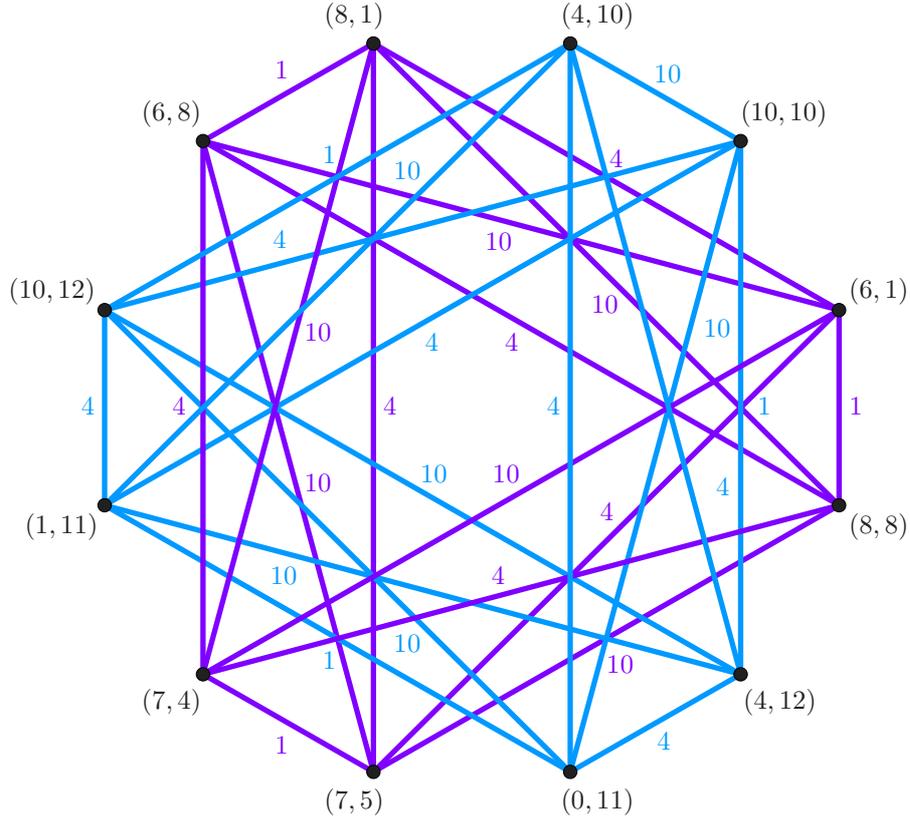
\begin{figure}[h]  
\begin{center}
\pagestyle{empty}

\definecolor{qqzzff}{rgb}{0.,0.6,1.}
\definecolor{qqqqff}{rgb}{0.,0.,1.}
\definecolor{sqsqsq}{rgb}
{0.12549019607843137,0.12549019607843137,0.12549019607843137}
\definecolor{xfqqff}{rgb}{0.4980392156862745,0.,1.}
\begin{tikzpicture}[line cap=round,line join=round,>=triangle 45,x=0.5cm,y=0.5cm]
\clip(-13.,-12.) rectangle (13.,12.);
\draw [line width=2.pt,color=xfqqff] (-2.5881904510252,-9.659258262890663)-- (9.659258262890685,-2.588190451025199);
\draw [line width=2.pt,color=xfqqff] (9.659258262890685,-2.588190451025199)-- (9.659258262890681,2.588190451025214);
\draw [line width=2.pt,color=xfqqff] (-2.588190451025208,9.659258262890681)-- (-7.0710678118654755,7.071067811865474);
\draw [line width=2.pt,color=xfqqff] (-7.0710678118654755,7.071067811865474)-- (-2.5881904510252,-9.659258262890663);
\draw [line width=2.pt,color=xfqqff] (-2.588190451025208,9.659258262890681)-- (-2.5881904510252,-9.659258262890663);
\draw [line width=2.pt,color=xfqqff] (9.659258262890681,2.588190451025214)-- (-2.5881904510252,-9.659258262890663);
\draw [line width=2.pt,color=xfqqff] (-7.0710678118654755,7.071067811865474)-- (9.659258262890685,-2.588190451025199);
\draw [line width=2.pt,color=xfqqff] (-7.0710678118654755,7.071067811865474)-- (9.659258262890681,2.588190451025214);
\draw [line width=2.pt,color=xfqqff] (-2.588190451025208,9.659258262890681)-- (9.659258262890685,-2.588190451025199);
\draw [line width=2.pt,color=xfqqff] (-2.588190451025208,9.659258262890681)-- (9.659258262890681,2.588190451025214);
\draw [line width=2.pt,color=qqzzff] (2.5881904510252163,-9.659258262890681)-- (7.071067811865483,-7.071067811865469);
\draw [line width=2.pt,color=qqzzff] (2.5881904510252163,-9.659258262890681)-- (7.071067811865475,7.0710678118654755);
\draw [line width=2.pt,color=qqzzff] (2.5881904510252163,-9.659258262890681)-- (-9.659258262890681,2.5881904510252096);
\draw [line width=2.pt,color=qqzzff] (2.5881904510252163,-9.659258262890681)-- (-9.659258262890683,-2.588190451025205);
\draw [line width=2.pt,color=qqzzff] (7.071067811865483,-7.071067811865469)-- (-9.659258262890683,-2.588190451025205);
\draw [line width=2.pt,color=qqzzff] (7.071067811865483,-7.071067811865469)-- (-9.659258262890681,2.5881904510252096);
\draw [line width=2.pt,color=qqzzff] (7.071067811865483,-7.071067811865469)-- (7.071067811865475,7.0710678118654755);
\draw [line width=2.pt,color=qqzzff] (7.071067811865475,7.0710678118654755)-- (-9.659258262890683,-2.588190451025205);
\draw [line width=2.pt,color=qqzzff] (7.071067811865475,7.0710678118654755)-- (-9.659258262890681,2.5881904510252096);
\draw [line width=2.pt,color=qqzzff] (-9.659258262890681,2.5881904510252096)-- (-9.659258262890683,-2.588190451025205);
\draw [line width=2.pt,color=xfqqff] (-2.5881904510252,-9.659258262890663)-- (-7.071067811865477,-7.071067811865472);
\draw [line width=2.pt,color=xfqqff] (9.659258262890685,-2.588190451025199)-- (-7.071067811865477,-7.071067811865472);
\draw [line width=2.pt,color=xfqqff] (9.659258262890681,2.588190451025214)-- (-7.071067811865477,-7.071067811865472);
\draw [line width=2.pt,color=xfqqff] (-2.588190451025208,9.659258262890681)-- (-7.071067811865477,-7.071067811865472);
\draw [line width=2.pt,color=xfqqff] (-7.0710678118654755,7.071067811865474)-- (-7.071067811865477,-7.071067811865472);
\draw [line width=2.pt,color=qqzzff] (2.5881904510252163,-9.659258262890681)-- (2.5881904510252074,9.659258262890685);
\draw [line width=2.pt,color=qqzzff] (7.071067811865483,-7.071067811865469)-- (2.5881904510252074,9.659258262890685);
\draw [line width=2.pt, color=qqzzff] (7.071067811865475,7.0710678118654755)-- (2.5881904510252074,9.659258262890685);
\draw [line width=2.pt,color=qqzzff] (-9.659258262890683,-2.588190451025205)-- (2.5881904510252074,9.659258262890685);
\draw [line width=2.pt,color=qqzzff] (-9.659258262890681,2.5881904510252096)-- (2.5881904510252074,9.659258262890685);
\begin{scriptsize}
\draw [fill=sqsqsq] (2.5881904510252074,9.659258262890685) circle (2.5pt);
\draw[color=sqsqsq] (3.2,10.4) node {\fontsize{10}{0} $ \left( 4,10 \right)$};
\draw[color=sqsqsq] (-3.2,10.4) node {\fontsize{10}{0} $ \left( 8,1 \right)$};
\draw[color=sqsqsq] (-3.2,-10.45) node {\fontsize{10}{0} $ \left( 7,5 \right)$};
\draw[color=sqsqsq] (3.2,-10.45) node {\fontsize{10}{0} $ \left( 0,11 \right)$};
\draw[color=sqsqsq] (-8.,7.8) node {\fontsize{10}{0} $ \left( 6,8 \right)$};
\draw[color=sqsqsq] (8.1,7.8) node {\fontsize{10}{0} $ \left( 10,10 \right)$};
\draw[color=sqsqsq] (-8.,-7.8) node {\fontsize{10}{0} $ \left( 7,4 \right)$};
\draw[color=sqsqsq] (8.0,-7.8) node {\fontsize{10}{0} $ \left( 4,12 \right)$};
\draw[color=sqsqsq] (-11.15,3.1) node {\fontsize{10}{0} $ \left( 10,12 \right)$};
\draw[color=sqsqsq] (10.6,3.1) node {\fontsize{10}{0} $ \left( 6,1 \right)$};
\draw[color=sqsqsq] (-10.9,-3.15) node {\fontsize{10}{0} $ \left( 1,11 \right)$};
\draw[color=sqsqsq] (10.6,-3.15) node {\fontsize{10}{0} $ \left( 8,8 \right)$};
\draw [fill=sqsqsq] (7.071067811865475,7.0710678118654755) circle (2.5pt);
\draw [fill=sqsqsq] (-2.588190451025208,9.659258262890681) circle (2.5pt);
\draw [fill=sqsqsq] (-7.0710678118654755,7.071067811865474) circle (2.5pt);
\draw [fill=sqsqsq] (-9.659258262890683,-2.588190451025205) circle (2.5pt);
\draw [fill=sqsqsq] (7.071067811865483,-7.071067811865469) circle (2.5pt);
\draw [fill=sqsqsq] (9.659258262890681,2.588190451025214) circle (2.5pt);
\draw [fill=sqsqsq] (-9.659258262890681,2.5881904510252096) circle (2.5pt);
\draw [fill=sqsqsq] (-7.071067811865477,-7.071067811865472) circle (2.5pt);
\draw [fill=sqsqsq] (2.5881904510252163,-9.659258262890681) circle (2.5pt);
\draw [fill=sqsqsq] (9.659258262890685,-2.588190451025199) circle (2.5pt);
\draw [fill=sqsqsq] (-2.5881904510252,-9.659258262890663) circle (2.5pt);
\draw[color=xfqqff] (-7.8,0.05) node {\fontsize{10}{0} $ 4 $};
\draw[color=qqzzff] (-10.2,0.05) node {\fontsize{10}{0} $ 4 $};
\draw[color=qqzzff] (2.05,0.05) node {\fontsize{10}{0} $ 4 $};
\draw[color=qqzzff] (7.6,0.05) node {\fontsize{10}{0} $ 1 $};
\draw[color=xfqqff] (-2.25,0.05) node {\fontsize{10}{0} $ 4 $};
\draw[color=xfqqff] (10.0,0.05) node {\fontsize{10}{0} $ 1 $};
\draw[color=qqzzff] (-1.8,-6.25) node {\fontsize{10}{0} $ 10 $};
\draw[color=qqzzff] (-1.8,6.3) node {\fontsize{10}{0} $ 10 $};
\draw[color=xfqqff] (0.6,4.4) node {\fontsize{10}{0} $ 10 $};
\draw[color=xfqqff] (0.6,-4.45) node {\fontsize{10}{0} $ 4 $};
\draw[color=xfqqff] (0.95,1.75) node {\fontsize{10}{0} $ 4 $};
\draw[color=xfqqff] (0.8,-1.75) node {\fontsize{10}{0} $ 10 $};
\draw[color=qqzzff] (-1.15,1.75) node {\fontsize{10}{0} $ 4 $};
\draw[color=qqzzff] (-1.1,-1.75) node {\fontsize{10}{0} $ 10 $};
\draw[color=xfqqff] (3.8,-6.8) node {\fontsize{10}{0} $ 10 $};
\draw[color=xfqqff] (3.7,6.6) node {\fontsize{10}{0} $ 4 $};
\draw[color=xfqqff] (3.45,-2.75) node {\fontsize{10}{0} $ 4 $};
\draw[color=xfqqff] (3.4,2.75) node {\fontsize{10}{0} $ 10 $};
\draw[color=qqzzff] (-3.85,-6.7) node {\fontsize{10}{0} $ 1 $};
\draw[color=qqzzff] (-3.85,6.7) node {\fontsize{10}{0} $ 1 $};
\draw[color=xfqqff] (-4.15,-2.) node {\fontsize{10}{0} $ 10 $};
\draw[color=xfqqff] (-4.15,2.) node {\fontsize{10}{0} $ 10 $};
\draw[color=qqzzff] (6.5,-2.1) node {\fontsize{10}{0} $ 4 $};
\draw[color=qqzzff] (6.35,2.1) node {\fontsize{10}{0} $ 10 $};
\draw[color=qqzzff] (5.05,8.85) node {\fontsize{10}{0} $ 10 $};
\draw[color=qqzzff] (4.95,-8.85) node {\fontsize{10}{0} $ 4 $};
\draw[color=xfqqff] (-5.1,8.95) node {\fontsize{10}{0} $ 1 $};
\draw[color=xfqqff] (-5.1,-8.95) node {\fontsize{10}{0} $ 1 $};
\draw[color=qqzzff] (-5.05,-4.45) node {\fontsize{10}{0} $ 10 $};
\draw[color=qqzzff] (-5.15,4.45) node {\fontsize{10}{0} $ 4 $};
\end{scriptsize}
\end{tikzpicture}

\caption{The points of $ C \left((7,11),6 \right)_{\mathbb{F}_{13}} $ and their rational squared distances} 
\label{F_13_with_distances}
\end{center}
\end{figure}
%%%%%%%%%%%%%%%%%%%%%%%%%%%%%%%%%%%%%%%%%%%%%%%%%	

If we compare \Cref{F_7_with_distances} with \Cref{F_13_with_distances}, we might expect that the c-maximal circular point sets on the same circle in prime field planes are always two disjoint subsets of equal cardinality. We will see later that this is true. Additionally, we see that the occurring squared distances from one point to the others in such c-maximal circular point set seem to be the same. To describe this observation we would like to introduce the next definition followed by an example and a non-example of other algebraic curves.

\begin{definition}
	An algebraic curve in a plane over a field $ \mathbb{F} $ has the {\em uniformity property} if for any two points $ P $ and $ Q $ on the curve and for any $ n \in \mathbb{N} $ the following holds true:
	If $ d \in \mathbb{F} $ is a squared distance between $ P $ and $ n $ different points of the curve, then there exists also $ n $ different points on the same curve which have squared distance $ d $ to $ Q $.
\end{definition}
	
\begin{example} \label{ex_uniform_prop}
	The parabola in the finite plane $ \mathbb{F}_{5} \times \mathbb{F}_{5} $ defined by the equation $ y = x^2 $ does not have the uniformity property: there are five different points on this parabola and the squared distance $ 1 $ only occur between one pair of these points. Hence, the uniformity property is not satisfied by this parabola.
\newline
Whereas for the line $ y = ax + b $ it seems to be intuitively clear that the uniformity property must hold over every field plane $ \mathbb{F} $ where $ a \in \mathbb{F}^{*} $ and $ b \in \mathbb{F} $. Let $ \mathbb{F} $ be arbitrary and $ P = \left( p_1,p_2 \right) $, $ Q = \left( q_1,q_2 \right) $, $ R_i = \left( x_i,y_i \right) $ be different points on this line for $ i = 1,2, \dots, n $ with $ n \in \mathbb{N} $. Assume that the squared distances between $ P $ and all these $ R_i $'s is $ d \in \mathbb{F} $. Define $ S_i \coloneqq \left( x_i + q_1 - p_1,y_i + q_2 - p_2 \right) $ for $ i = 1,2,...,n $, then these points are on the same line and must be different from each other with squared distance $ d $ to $ Q $ which shows the uniformity property.
\end{example}

Soon we are ready to prove that every circle on an arbitrary field plane has the uniformity property. For this we will work with the set $ \Theta_{M}^{\mathbb{F}} $ for $ M = \left( 0,0 \right) $.

%\begin{definition} \label{rot_matrix}
%	Let $ \mathbb{F} $ be any field. Then we define
%	\begin{displaymath}
%		\Theta_{\mathbb{F}} \coloneqq %\Set{\begin{pmatrix}
% a & -b \\
% b & a \\
%\end{pmatrix} \in \mathbb{F}^{2 \times 2} | a^2 + b^2 = 1} 
%	\end{displaymath}
%and call it the {\em rotation matrix} group of $ \mathbb{F} $.	
%\end{definition}

%Observe that the set $ \Theta_{\mathbb{F}} $ equipped with the matrix multiplication defines an abelian group. %Now we are ready to show that every circle satisfies the uniformity property.

\begin{lemma} \label{lemma_uniform_prop}
	Every circle in an arbitrary field plane satisfies the uniformity property.
\end{lemma}

\begin{proof}
	Without loss of generality, we can consider a circle centered at the origin in a plane over an arbitrary field $ \mathbb{F} $ because squared distances are preserved when we shift a circle by \Cref{squared_distance_invariant}. Moreover, we can assume that this circle has radius $ 1 $. This follows from the fact that scaling the radius with a non-vanishing factor in a prime field also scales all the mutual squared distances of the points by the same squared factor. Hence, the scaled circle has the uniformity property if and only if the original circle does so.

Recall that the group $ \Theta_{\left( 0,0 \right)}^{\mathbb{F}} $ acts on the plane $ \mathbb{F} \times \mathbb{F} $ such that squared distances between points are preserved, see \Cref{squared_distance_invariant}. %, i.e. any two points and the corresponding image points have the same distance from each other, respectively. 
Moreover, it is easy to see that $ \Theta_{\left( 0,0 \right)}^{\mathbb{F}} $ restricted to the unit circle $ C \left(\left( 0,0 \right),1 \right)_{\mathbb{F}} $ permutes its points. Hence, it is enough to show that $\Theta_{\left( 0,0 \right)}^{\mathbb{F}} $ acts transitively on the plane $ \mathbb{F} \times \mathbb{F} $ restricted to $ C \left(\left( 0,0 \right),1 \right)_{\mathbb{F}} $ because then the uniformity property is clearly satisfied. 
% due to the fact that squared distances are preserved. 
Thus, consider two points
	\begin{align*}
		\underbrace{ \left(p_1,p_2 \right)}_{=: P}, \underbrace{\left( q_1,q_2 \right)}_{=: Q} \in C \left(\left( 0,0 \right),1 \right)_{\mathbb{F}}.
	\end{align*}
%Then we get that the matrix 
%	\begin{displaymath}
%		M_{P,Q} \coloneqq 
%\begin{pmatrix}
% q_1 & -q_2 \\
% q_2 & q_1 \\
%\end{pmatrix}
%\begin{pmatrix}
% p_1 & -p_2 \\
% p_2 & p_1 \\
%\end{pmatrix}^{-1} \in \Theta_{\mathbb{F}}
%	\end{displaymath}
% maps the point $ P $ to $ Q $ and so we are done.	
Then we get that  
	$$ \left( \theta^{\left( q_1,q_2 \right)} \right) \circ \left(\theta^{\left( p_1,p_2 \right)}\right)^{-1} \in \Theta_{\left( 0,0 \right)}^{\mathbb{F}} $$
maps the point $ P $ to $ Q $ and so we are done.
\end{proof}

The last result can be used to prove the following.
\begin{proposition} \label{prop_fin_prim_field}
	Let $ \mathbb{F} $ be a finite prime field with $ \chr \left( \mathbb{F} \right) \neq 2 $. Then all points of any circle $ C \left((m_1,m_2),r \right)_{\mathbb{F}} $ in the affine plane $ \mathbb{F} \times \mathbb{F} $ can be divided into two disjoint e-maximal circular point sets of cardinality either $ \frac{ \vert \mathbb{F} \vert +1}{2} $ if $ \sqrt{-1} \notin \mathbb{F} $ or $ \frac{\vert \mathbb{F} \vert -1}{2} $ otherwise. 
	%Moreover, the e-maximal circular point sets over finite prime field planes, where the characteristic of the underlying prime field is odd, are c-maximal.
\end{proposition}
In case $ \chr \left( \mathbb{F} \right) = 2 $ we trivially have that all points of a circle are an e-maximal and so also a c-maximal circular point set.
%are a c-maximal circular point set by \Cref{dist_lemma}. So this case is already treated.

\begin{proof}
At first we would like to show that all points of circles over these finite prime field planes can be partitioned into two groups of point sets such that all the points in each group have rational squared distance from each other, but no two points from different groups have r.s.d..
	%As a first step we would like to show that all points of circles over these finite prime field planes consists of exactly two maximal  circular point sets. To show this, it is enough to show that we can split up all the points of an arbitrary circle $ C \left(r,(m_1,m_2) \right)_{\mathbb{F}} \subset \mathbb{F} \times \mathbb{F} $ into two groups such that all the points in each group have rational squared distance from each other and no two points from different groups have are at rational squared distance. \newline
Consider the following map
	\begin{align*}
\mathrm{sign}: \mathbb{F}^{*} &\rightarrow \Set{-1,1} \\
   a &\mapsto \left\{
\begin{array}{lll}
1 &\textrm{if} \ a \in \square_{\mathbb{F}^{*}} \\
-1 &  \textrm{otherwise} \\
\end{array}
\right. .
	\end{align*}
Since $ \square_{\mathbb{F}^{*}} \subseteq \mathbb{F}^{*} $ is a subgroup of index $ 2 $ with respect to multiplication, we conclude that $ \sign $ is a group homomorphism. 

By writing $ P_t $ for the point parametrized by $ t $, as introduced before \Cref{coro_parametrisation}, we set:	
	\begin{align*}
		C_1 &\coloneqq \Set{ P_t \in C \left((m_1,m_2),r \right)_{\mathbb{F}} | t^2 + 1 \notin \square_{\mathbb{F}^{*}} } \\
		C_2 &\coloneqq \ \Set{ P_t \in C \left((m_1,m_2),r \right)_{\mathbb{F}} | t^2 + 1 \in \square_{\mathbb{F}^{*}} } \cup \Set{ \left( m_1 , m_2 + r \right) }
	\end{align*}	
	
Let $ t_1, t_2 \in \mathbb{F} $ be two different parameters which are not a square root of $ -1 $. Then we can deduce
	\begin{align*}
		D^2 \left( P_1, P_2 \right) \in \mathbb{F}^2 &\Longleftrightarrow \left( t_1^2 + 1 \right) \left( t_2^2 + 1\right) \in \square_{\mathbb{F}^{*}} \\
		D^2 \left( P_1, \left( m_1 , m_2 + r \right) \right) \in \mathbb{F}^2 &\Longleftrightarrow \left( t_1^2 + 1 \right) \in \square_{\mathbb{F}^{*}}
	\end{align*}
by \Cref{dist_lemma}. 

Hence, by the above and the fact that $ \sign $ is a group homomorphism we can conclude that all the points in $ C_1 $ and $ C_2 $ have rational squared distance from each other, respectively, but there are no points from different groups which have r.s.d.. This also shows that $ C_1 $ and $ C_2 $ are the only e-maximal circular point sets and they are disjoint by definition. By the uniformity property of the circle we can deduce that disjoint e-maximal circular point sets must have the same cardinality, i.e. $ C_1 $ and $ C_2 $ have the same cardinality and so we can conclude by \Cref{coro_parametrisation}.
\end{proof}

% By \Cref{lemma_uniform_prop} each point of $ C \left(r,(m_1,m_2) \right)_{\mathbb{F}} $ has the same squared distances to the other points of the same circle. Hence, we can conclude that the cardinality of these maximal circular point sets must be equal and the number of this cardinality follows then by \Cref{coro_parametrisation}.
%\newline
%The last sentence follows from the fact, that all e-maximal circular point sets have the same cardinality.

\begin{example}
	Consider again \Cref{F_7_with_distances} and 	\Cref{F_13_with_distances}. By \Cref{prop_fin_prim_field} each circle in $ \mathbb{F}_7 $ consists of two disjoint c-maximal circular point sets of cardinality $ 4 $ and in $ \mathbb{F}_{13} $ the cardinality of the c-maximal circular point sets is $ 6 $.
\end{example}

Now we would like to consider the case if $ \mathbb{F} $ is a prime field with $ \chr \left( \mathbb{F} \right) = \infty $. Then $ \mathbb{F} $ must be isomorphic to $ \mathbb{Q} $. Hence, the notion “rational” becomes more intuitive since a squared distance between two points on a circle in the plane over $ \mathbb{Q} $ is rational if and only if the Euclidean distance between these points is rational.
Observe that for finite prime fields the number of maximal circular point sets was equal to the index $ \left[ \mathbb{F}^{*} : \square_{\mathbb{F}^{*}} \right] $ which is  $ 1 $ if and only if $ \chr \left( \mathbb{F} \right) = 2 $ and otherwise $ 2 $. %This was indirectly used in the proof of \Cref{prop_fin_prim_field} as we defined the two sets $ C_1 $ and $ C_2 $. 
% Later we will see that this is true for all circles over any field. 
Before we go on, let us prove the following useful Lemma which gives us the index in the case that $ \mathbb{F} $ is infinite.

\begin{lemma} \label{lemma_cosets_of_Q}
	Let the set of primes in $ \mathbb{N} $ denote by $ \mathbb{P} $. We can decompose the multiplicative group $ \mathbb{Q}^{*} $ into the following cosets with respect to the subgroup $ \square_{\mathbb{Q}^{*}} $:
		\begin{displaymath}
		\mathbb{Q}^{*} = \coprod_{\substack{n \in \mathbb{N}, k \in \left\{ 0,1 \right\} \\ p_1,p_2, \ldots, p_n \in \mathbb{P} \\ p_1 < p_2 < \ldots < p_n}}{ \hspace{-1.5 em} \left( -1 \right)^{k} p_1 p_2 \ldots p_n \square_{\mathbb{Q}^{*}}}.
		\end{displaymath}
Moreover, we get $ \left[ \mathbb{Q}^{*} : \square_{\mathbb{Q}^{*}} \right] = \infty $ (countably infinite).		
\end{lemma}

\begin{proof}
	The last statement of \Cref{lemma_cosets_of_Q} follows directly from the decomposition of $ \mathbb{Q}^{*} $ into cosets and the fact that the number of all finite subsets of countably infinite set (here the set of all prime numbers $ \mathbb{P} $) must also be countably infinite, see \cite{Halb_Comb_Set}. Hence, it remains to show that the above decomposition of $ \mathbb{Q}^{*} $ is a decomposition of cosets. In particular, it remains to show that $ \mathbb{Q}_{+}^{*} $ can be decomposed in the positive cosets above. Consider an element $ q \in \mathbb{Q}_{+}^{*} $, then we can apply the fundamental theorem of arithmetic which gives us $ n \in \mathbb{N} $, $ m_i \in \mathbb{Z} \setminus \left\{ 0 \right\} \ \forall i = 1,2, \ldots, m $ and primes $ p_1 < p_2 < \ldots < p_n $ such that $ q =  \prod \limits_{i=1}^{n}{p_i^{m_i}} $. Then there are $ l \in \mathbb{N} $ indices $ i $ such that the numbers $ n_i $ are odd. We enumerate these indices by $ i_j $ $ \forall j = 1,2,\ldots l $. Thus, we clearly have that 
	$$ q \in \prod \limits_{j=1}^{l}{p_{i_j}} \square_{ \mathbb{Q}^{*}} .$$

To finish the proof we only have to show that all the positive cosets are different from each other. Therefore let $ n,m \in \mathbb{N} $, $ p_1, p_2, \ldots, p_n$ and $ q_1,q_2, \ldots, q_m $ be two sequences of primes where in each sequence all primes are different and increasing such that
	\begin{displaymath}
		p_1 p_2 \ldots p_n \square_{\mathbb{Q}_{+}^{*} } = q_1 q_2 \ldots q_m \square_{\mathbb{Q}_{+}^{*} }.
	\end{displaymath} 
Hence, we must have that the product of all these prime numbers must be in $ \square_{\mathbb{Q}_{+}^{*} } $, so if we decompose this product into prime factors, then their exponents must be even which implies that there are all equal to $ 2 $, $ n = m $ and $ p_i = q_i \ $ for all $ i = 1,2, \ldots, n $.
\end{proof}

% With this we can come to the following statement.

\begin{proposition} \label{prop_infin_prim_field}
	The e-maximal circular point sets of a circle over $ \mathbb{Q} \times \mathbb{Q} $ are disjoint sets of (countably) infinite cardinality.
\end{proposition}

\begin{proof}
	Let $ C \left((m_1,m_2),r \right)_{\mathbb{Q}} $ be an arbitrary circle over the rational plane $ \mathbb{Q} \times \mathbb{Q} $ and recall that there are no two points on this circle with vanishing squared distance from each other by \Cref{coro_no_vanish}. If $ P_1 $ and $ P_2 $ are parametrized by $ t_1, t_2 \in \mathbb{Q} $, respectively, then 
	\begin{align*}
D^2 \left( P_1, P_2 \right) \in \square_{\mathbb{Q}^{*}} &\Longleftrightarrow \left( t_1^2 + 1 \right) \left( t_2^2 + 1\right) \in \square_{\mathbb{Q}^{*}} \\
		D^2 \left( P_1, \left( m_1 , m_2 + r \right) \right) \in \square_{\mathbb{Q}^{*}} &\Longleftrightarrow \left( t_1^2 + 1 \right) \in \square_{\mathbb{Q}^{*}}
	\end{align*}
where $ \left( t_1^2 + 1 \right) \left( t_2^2 + 1\right) \in \square_{\mathbb{Q}^{*}} $ is equivalent to the fact that $ t_1^2 + 1 $ and $ t_2^2 + 1 $ belong to the same coset which follows by a similar argument as in the last part of the proof of \Cref{lemma_cosets_of_Q}. This shows that the e-maximal circular point sets of $ C \left((m_1,m_2),r \right)_{\mathbb{Q}} $ are equal to the unions of all points of $ C \left((m_1,m_2),r \right)_{\mathbb{Q}} $ parametrized by a $ t \in \mathbb{Q} $ such that $ t^2 + 1 $ is in the same coset where $ \left( m_1 , m_2 + r \right) $ belongs to the coset $ \square_{\mathbb{Q}^{*} } $. And the cardinalities of each such circular point set are equal by the uniformity property of the circle. 

It remains to show that one of these cosets contains infinitely many elements. For this let us show that there exists countably infinitely many $ t \in \mathbb{Q} $ such that $ t^2 + 1 \in \square_{\mathbb{Q}} $ or more concretely: Let us show that there exists infinitely many $ t,u \in \mathbb{Q} $ such that
	\begin{displaymath}
		t^2 + 1 = u^2 \Longleftrightarrow \left(t+u \right) \left(t-u \right) = -1
	\end{displaymath}
holds true.	
For this let us assume that we have $ n \in \mathbb{N} \setminus \left\{ 0 \right\} $ with
	\begin{align*}
		t + u &= n \\
		t - u &= -\frac{1}{n}.
	\end{align*}
Then we clearly see that we can solve for $ t $ and $ u $ in $ \mathbb{Q} $ such that we get
	\begin{align*}
		t = \dfrac{1}{2} \left( n - \frac{1}{n} \right) && u = \dfrac{1}{2} \left( n + \frac{1}{n} \right)
	\end{align*}
which satisfies the conditions above. Moreover, it is easy to see that for all $ n \in \mathbb{N} \setminus \left\{ 0 \right\} $, the resulting $ t $'s are different. %which shows that there exists countably infinitely many such $ t $'s and 
Therefore each e-maximal circular point set in $ C \left((m_1,m_2),r \right)_{\mathbb{Q}} $ is of countably infinite cardinality.
\end{proof}

This means that the rational points, i.e. the points with rational coordinates of every circle in the real plane with non-vanishing and rational radius can be decomposed in countably many sets of countably many points such that in each set any two points have r.s.d..
A proof that the c-maximal circular point sets in $ \mathbb{R} \times \mathbb{R} $ are countably infinite and dense in the underlying circle with respect to the Euclidean topology can be done by the use of perfect distances, see \cite[Proposition 2.45, Lemma 2.47]{my_masterthesis}. In the next session we will generalize the notion of perfect distances and use it as a tool for the construction of e- and c-maximal circular point sets in all planes over any field extension. %It is also possible to show that all these e-maximal circular point sets are countably infinite and dense in their containing circle in the Euclidean plane, see \cite[]{my_masterthesis}. 

Finally, we can conclude the following result by \Cref{dist_lemma}, \Cref{prop_fin_prim_field}, \Cref{lemma_cosets_of_Q} and \Cref{prop_infin_prim_field} which summarizes the whole section. 

\begin{theorem}[e-maximal circular point sets over prime field planes] \label{theorem_prime_field}
	Let $ \mathbb{F} $ be any prime field. % with $ \chr \left( \mathbb{F} \right) \neq 2 $. 
Then all circles in $ \mathbb{F} \times \mathbb{F} $ consists of exactly $ \left[ \mathbb{F}^{*} : \square_{ \mathbb{F}^{*}} \right]  $ different e-maximal circular point sets which are disjoint and of cardinality
	\begin{align*} 
	\left\{
\begin{array}{lll} 
%%%% \hspace{0.35cm} 
\hspace{0.05cm} 
\lvert \mathbb{F} \rvert &  \textrm{if} \ \chr \left( \mathbb{F} \right) = 2 \\ [2.5ex]
\dfrac{\lvert \mathbb{F} \rvert -1}{2} & \textrm{if} \ \chr \left( \mathbb{F} \right) \neq 2 \ \textrm{and} \ \sqrt{-1} \in \mathbb{F} \\ [2.5ex]
\dfrac{\lvert \mathbb{F} \rvert + 1}{2} &  \textrm{otherwise} \\
\end{array}
\right. .
	\end{align*}		
\end{theorem}

Note that all e-maximal circular point sets over a prime field plane have the same cardinality and define a partition on the point of the underlying circle. Hence, over prime field planes there is no difference between e- and c-maximal circular point sets. However, we will see that this is not true if we work with circles over arbitrary field planes.
\end{subsection}

%Observe that the prime field $ \mathbb{F}_7 $ contains no root of $ -1 $ whereas $ \mathbb{F}_{13} $ does. Hence, the above formulas for the e-maximal point sets holds true as we can see in \Cref{F_7_with_distances} and \Cref{F_13_with_distances}.

% The definiton of congruent for circular point sets

\begin{subsection}{Perfect distances}

In the last section we have seen that the points of circles over finite fields can be divided into equivalence classes by defining a relation  “has rational distance to", i.e. for this relation, points on a circle are related to each other if and only if the squared distance between them is rational. However, the next example will show that for circles over general fields this is clearly not the case. We will also assume that our field has a characteristic different from $ 2 $ in this section.

\begin{example} \label{ex_F7_extended}
	We would like to find the e-maximal circular point sets containing an arbitrary point of $ C \left((0,0),1 \right)_{\mathbb{F}_{49}} $. In order to do so, we identify the field $ \mathbb{F}_{49} $ with the isomorphic polynomial ring $ \mathbb{F}_{7}[a] $ where $ a $ is a root of the irreducible polynomial $ x^2 + 1 $ over $ \mathbb{F}_{7} $. As fixed point we choose $ \left( a+4,5a+2 \right) $. Since $ \boldsymbol{P} \left( \mathbb{F}_{49} \right) = \mathbb{F}_{7} $, the squared distances between points of the above circle are rational if and only if their values are contained in $ \square_{\mathbb{F}_{7}^{*}} = \{ 1,2,4 \} $. By calculating the squared distances between the points as in \Cref{existence_perf_dist} we see that there are exactly three e-maximal circular point sets containing our fixed point: All the points connected by solid blue lines and the two pairs of points connected by dashed blue lines. Moreover, only the four points connected with dark blue lines define a c-maximal circular point set containing our fixed point. By the uniformity property the same would also hold true for any other point in $ C \left((0,0),1 \right)_{\mathbb{F}_{49}} $.
\end{example}  

%%%%%%%%%%%%%%%%%%%%%%%%%%%%%%%%%%%%%%%%%%%%%%%%%	
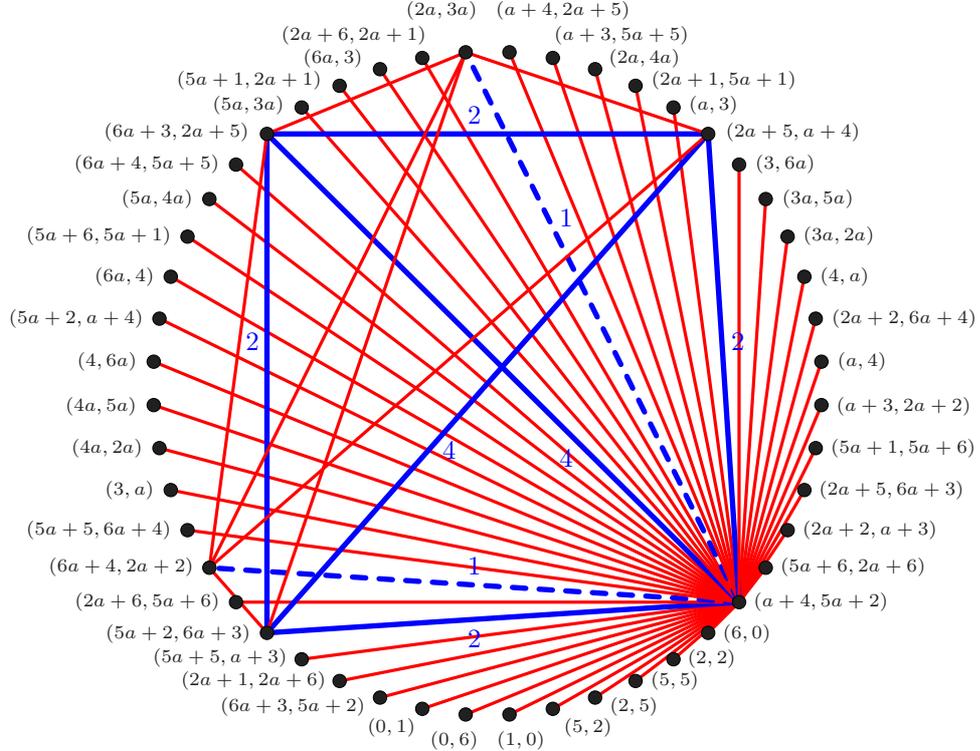
\begin{figure}[h]  
\begin{center}
\pagestyle{empty}

\definecolor{qqzzff}{rgb}{0.,0.6,1.}
\definecolor{qqqqff}{rgb}{0.,0.,1.}
\definecolor{ffqqqq}{rgb}{1.,0.,0.}
\definecolor{sqsqsq}{rgb}{0.12549019607843137,0.12549019607843137,0.12549019607843137}
\begin{tikzpicture}[line cap=round,line join=round,>=triangle 45,x=1.1cm,y=1.1cm]
\clip(-6,-5) rectangle (6,5);
\draw [line width=1.2pt,color=ffqqqq] (3.007359229915902,-2.6373832604002847)-- (2.6373832604002634,-3.00735922991592);
\draw [line width=1.2pt,color=ffqqqq] (3.007359229915902,-2.6373832604002847)-- (3.325878449210176,-2.222280932078416);
\draw [line width=1.2pt,color=ffqqqq] (3.007359229915902,-2.6373832604002847)-- (3.587490966130751,-1.7691547608760103);
\draw [line width=1.2pt,color=ffqqqq] (3.007359229915902,-2.6373832604002847)-- (3.7877205179804214,-1.2857578612126495);
\draw [line width=1.2pt,color=ffqqqq] (3.007359229915902,-2.6373832604002847)-- (3.9231411216129213,-0.7803612880645155);
\draw [line width=1.2pt,color=ffqqqq] (3.007359229915902,-2.6373832604002847)-- (3.991435692954414,-0.26161251692057397);
\draw [line width=1.2pt,color=ffqqqq] (3.007359229915902,-2.6373832604002847)-- (3.991435692954415,0.26161251692057225);
\draw [line width=1.2pt,color=ffqqqq] (3.007359229915902,-2.6373832604002847)-- (3.9231411216129213,0.7803612880645147);
\draw [line width=1.2pt,color=ffqqqq] (3.007359229915902,-2.6373832604002847)-- (3.7877205179804214,1.285757861212651);
\draw [line width=1.2pt,color=ffqqqq] (3.007359229915902,-2.6373832604002847)-- (3.587490966130752,1.7691547608760065);
\draw [line width=1.2pt,color=ffqqqq] (3.007359229915902,-2.6373832604002847)-- (3.3258784492101827,2.222280932078407);
\draw [line width=1.2pt,color=ffqqqq] (3.007359229915902,-2.6373832604002847)-- (2.2222809320783883,-3.3258784492101947);
\draw [line width=1.2pt,color=ffqqqq] (3.007359229915902,-2.6373832604002847)-- (1.769154760875971,-3.58749096613077);
\draw [line width=1.2pt,color=ffqqqq] (3.007359229915902,-2.6373832604002847)-- (1.2857578612125975,-3.787720517980439);
\draw [line width=1.2pt,color=ffqqqq] (3.007359229915902,-2.6373832604002847)-- (0.780361288064448,-3.9231411216129346);
\draw [line width=1.2pt,color=ffqqqq] (3.007359229915902,-2.6373832604002847)-- (0.26161251692049253,-3.9914356929544197);
\draw [line width=1.2pt,color=ffqqqq] (3.007359229915902,-2.6373832604002847)-- (-0.2616125169206878,-3.9914356929544064);
\draw [line width=1.2pt,color=ffqqqq] (3.007359229915902,-2.6373832604002847)-- (-0.7803612880646181,-3.923141121612901);
\draw [line width=1.2pt,color=ffqqqq] (3.007359229915902,-2.6373832604002847)-- (-1.2857578612127412,-3.7877205179803903);
\draw [line width=1.2pt,color=ffqqqq] (3.007359229915902,-2.6373832604002847)-- (-1.769154760876088,-3.587490966130712);
\draw [line width=1.2pt,color=ffqqqq] (3.007359229915902,-2.6373832604002847)-- (-2.2222809320784793,-3.3258784492101334);
\draw [line width=2.pt,color=qqqqff] (3.007359229915902,-2.6373832604002847)-- (-2.637383260400333,-3.0073592299158594);
\draw [line width=1.2pt,color=ffqqqq] (3.007359229915902,-2.6373832604002847)-- (-3.0073592299159535,-2.6373832604002247);
\draw [line width=2.pt, dash pattern=on 5pt off 5pt, color=qqqqff] (3.007359229915902,-2.6373832604002847)-- (-3.3258784492102134,-2.222280932078361);
\draw [line width=1.2pt,color=ffqqqq] (3.007359229915902,-2.6373832604002847)-- (-3.5874909661307752,-1.7691547608759592);
\draw [line width=1.2pt,color=ffqqqq] (3.007359229915902,-2.6373832604002847)-- (-3.7877205179804374,-1.2857578612126033);
\draw [line width=1.2pt,color=ffqqqq] (3.007359229915902,-2.6373832604002847)-- (-3.9231411216129297,-0.7803612880644725);
\draw [line width=1.2pt,color=ffqqqq] (3.007359229915902,-2.6373832604002847)-- (-3.9914356929544166,-0.26161251692053566);
\draw [line width=1.2pt,color=ffqqqq] (3.007359229915902,-2.6373832604002847)-- (-3.991435692954412,0.26161251692060467);
\draw [line width=1.2pt,color=ffqqqq] (3.007359229915902,-2.6373832604002847)-- (-3.923141121612916,0.7803612880645409);
\draw [line width=1.2pt,color=ffqqqq] (3.007359229915902,-2.6373832604002847)-- (-3.7877205179804148,1.2857578612126697);
\draw [line width=1.2pt,color=ffqqqq] (3.007359229915902,-2.6373832604002847)-- (-3.587490966130744,1.769154760876023);
\draw [line width=1.2pt,color=ffqqqq] (3.007359229915902,-2.6373832604002847)-- (-3.3258784492101725,2.2222809320784216);
\draw [line width=1.2pt,color=ffqqqq] (3.007359229915902,-2.6373832604002847)-- (-3.007359229915903,2.6373832604002816);
\draw [line width=2.pt,color=qqqqff] (3.007359229915902,-2.6373832604002847)-- (-2.637383260400274,3.0073592299159095);
\draw [line width=1.2pt,color=ffqqqq] (3.007359229915902,-2.6373832604002847)-- (-2.2222809320784105,3.3258784492101796);
\draw [line width=1.2pt,color=ffqqqq] (3.007359229915902,-2.6373832604002847)-- (-1.7691547608760072,3.587490966130752);
\draw [line width=1.2pt,color=ffqqqq] (3.007359229915902,-2.6373832604002847)-- (-1.2857578612126472,3.7877205179804223);
\draw [line width=1.2pt,color=ffqqqq] (3.007359229915902,-2.6373832604002847)-- (-0.7803612880645124,3.9231411216129217);
\draw [line width=2.pt, dash pattern=on 5pt off 5pt, color=qqqqff] (3.007359229915902,-2.6373832604002847)-- (-0.26161251692057075,3.991435692954414);
\draw [line width=1.2pt,color=ffqqqq] (3.007359229915902,-2.6373832604002847)-- (0.26161251692057513,3.9914356929544135);
\draw [line width=1.2pt,color=ffqqqq] (3.007359229915902,-2.6373832604002847)-- (0.7803612880645163,3.923141121612921);
\draw [line width=1.2pt,color=ffqqqq] (3.007359229915902,-2.6373832604002847)-- (1.2857578612126495,3.7877205179804214);
\draw [line width=1.2pt,color=ffqqqq] (3.007359229915902,-2.6373832604002847)-- (1.769154760876009,3.5874909661307517);
\draw [line width=1.2pt,color=ffqqqq] (3.007359229915902,-2.6373832604002847)-- (2.2222809320784163,3.325878449210176);
\draw [line width=2.pt,color=qqqqff] (3.007359229915902,-2.6373832604002847)-- (2.6373832604002816,3.007359229915904);
\draw [line width=1.2pt,color=ffqqqq] (3.007359229915902,-2.6373832604002847)-- (3.0073592299159144,2.63738326040027);
\draw [line width=2.pt,color=qqqqff] (-2.637383260400274,3.0073592299159095)-- (2.6373832604002816,3.007359229915904);
\draw [line width=2.pt,color=qqqqff] (-2.637383260400333,-3.0073592299158594)-- (-2.637383260400274,3.0073592299159095);
\draw [line width=2.pt,color=qqqqff] (-2.637383260400333,-3.0073592299158594)-- (2.6373832604002816,3.007359229915904);
\draw [line width=1.2pt,color=ffqqqq] (-0.26161251692057075,3.991435692954414)-- (-3.3258784492102134,-2.222280932078361);
\draw [line width=1.2pt,color=ffqqqq] (-0.26161251692057075,3.991435692954414)-- (2.6373832604002816,3.007359229915904);
\draw [line width=1.2pt,color=ffqqqq] (-0.26161251692057075,3.991435692954414)-- (-2.637383260400274,3.0073592299159095);
\draw [line width=1.2pt,color=ffqqqq] (-0.26161251692057075,3.991435692954414)-- (-2.637383260400333,-3.0073592299158594);
\draw [line width=1.2pt,color=ffqqqq] (-3.3258784492102134,-2.222280932078361)-- (-2.637383260400333,-3.0073592299158594);
\draw [line width=1.2pt,color=ffqqqq] (-3.3258784492102134,-2.222280932078361)-- (-2.637383260400274,3.0073592299159095);
\draw [line width=1.2pt,color=ffqqqq] (-3.3258784492102134,-2.222280932078361)-- (2.6373832604002816,3.007359229915904);
\begin{scriptsize}
\draw [fill=sqsqsq] (3.991435692954415,0.26161251692057225) circle (2.5pt);
\draw[color=sqsqsq] (4.475,0.27) node {$ \left( a,4 \right) $};
\draw [fill=sqsqsq] (3.991435692954414,-0.26161251692057397) circle (2.5pt);
\draw[color=sqsqsq] (4.975,-0.275) node {$ \left( a + 3,2a + 2 \right) $};
\draw [fill=sqsqsq] (3.9231411216129213,0.7803612880645147) circle (2.5pt);
\draw[color=sqsqsq] (4.98,0.78) node {$ \left( 2a + 2,6a + 4 \right) $};
\draw [fill=sqsqsq] (3.9231411216129213,-0.7803612880645155) circle (2.5pt);
\draw[color=sqsqsq] (4.97,-0.775) node {$ \left( 5a + 1,5a + 6 \right) $};
\draw [fill=sqsqsq] (3.7877205179804214,1.285757861212651) circle (2.5pt);
\draw[color=sqsqsq] (4.27,1.285) node {$ \left( 4,a \right) $};
\draw [fill=sqsqsq] (3.7877205179804214,-1.2857578612126495) circle (2.5pt);
\draw[color=sqsqsq] (4.83,-1.29) node {$ \left( 2a + 5,6a + 3 \right) $};
\draw [fill=sqsqsq] (3.587490966130752,1.7691547608760065) circle (2.5pt);
\draw[color=sqsqsq] (4.2,1.775) node {$ \left( 3a,2a \right) $};
\draw [fill=sqsqsq] (3.587490966130751,-1.7691547608760103) circle (2.5pt);
\draw[color=sqsqsq] (4.57,-1.765) node {$ \left( 2a + 2,a + 3 \right) $};
\draw [fill=sqsqsq] (3.325878449210176,-2.222280932078416) circle (2.5pt);
\draw[color=sqsqsq] (4.37,-2.21) node {$ \left( 5a + 6,2a + 6 \right) $};
\draw [fill=sqsqsq] (3.0073592299159144,2.63738326040027) circle (2.5pt);
\draw[color=sqsqsq] (3.56,2.65) node {$ \left( 3,6a \right) $};
\draw [fill=sqsqsq] (3.007359229915902,-2.6373832604002847) circle (2.5pt);
\draw[color=sqsqsq] (3.985,-2.63) node {$ \left( a + 4,5a + 2 \right) $};
\draw [fill=sqsqsq] (3.3258784492101827,2.222280932078407) circle (2.5pt);
\draw[color=sqsqsq] (3.95,2.23) node {$ \left( 3a,5a \right) $};
\draw [fill=sqsqsq] (2.6373832604002816,3.007359229915904) circle (2.5pt);
\draw[color=sqsqsq] (3.65,3.025) node {$ \left( 2a + 5,a + 4 \right) $};
\draw [fill=sqsqsq] (2.6373832604002634,-3.00735922991592) circle (2.5pt);
\draw[color=sqsqsq] (3.1,-3.015) node {$ \left( 6,0 \right) $};
\draw [fill=sqsqsq] (2.2222809320784163,3.325878449210176) circle (2.5pt);
\draw[color=sqsqsq] (2.7,3.35) node {$ \left( a,3 \right) $};
\draw [fill=sqsqsq] (2.2222809320783883,-3.3258784492101947) circle (2.5pt);
\draw[color=sqsqsq] (2.68,-3.32) node {$ \left( 2,2 \right) $};
\draw [fill=sqsqsq] (1.769154760876009,3.5874909661307517) circle (2.5pt);
\draw[color=sqsqsq] (2.825,3.65) node {$ \left( 2a + 1,5a + 1 \right) $};
\draw [fill=sqsqsq] (1.769154760875971,-3.58749096613077) circle (2.5pt);
\draw[color=sqsqsq] (2.24,-3.58) node {$ \left( 5,5 \right) $};
\draw [fill=sqsqsq] (1.2857578612126495,3.7877205179804214) circle (2.5pt);
\draw[color=sqsqsq] (1.875,3.925) node {$ \left( 2a,4a \right) $};
\draw [fill=sqsqsq] (0.7803612880645163,3.923141121612921) circle (2.5pt);
\draw[color=sqsqsq] (1.6,4.2) node {$ \left( a + 3,5a + 5 \right) $};
\draw [fill=sqsqsq] (0.26161251692057513,3.9914356929544135) circle (2.5pt);
\draw[color=sqsqsq] (0.9,4.5) node {$ \left( a + 4,2a + 5 \right) $};
\draw [fill=sqsqsq] (-0.26161251692057075,3.991435692954414) circle (2.5pt);
\draw[color=sqsqsq] (-0.55,4.5) node {$ \left( 2a,3a \right) $};
\draw [fill=sqsqsq] (1.2857578612125975,-3.787720517980439) circle (2.5pt);
\draw[color=sqsqsq] (1.75,-3.9) node {$ \left( 2,5 \right) $};
\draw [fill=sqsqsq] (0.780361288064448,-3.9231411216129346) circle (2.5pt);
\draw[color=sqsqsq] (1.2,-4.15) node {$ \left( 5,2 \right) $};
\draw [fill=sqsqsq] (0.26161251692049253,-3.9914356929544197) circle (2.5pt);
\draw[color=sqsqsq] (0.4,-4.3) node {$ \left( 1,0 \right) $};
\draw [fill=sqsqsq] (-0.7803612880645124,3.9231411216129217) circle (2.5pt);
\draw[color=sqsqsq] (-1.6,4.2) node {$ \left( 2a + 6,2a + 1 \right) $};
\draw [fill=sqsqsq] (-1.2857578612126472,3.7877205179804223) circle (2.5pt);
\draw[color=sqsqsq] (-1.85,3.925) node {$ \left( 6a,3 \right) $};
\draw [fill=sqsqsq] (-1.7691547608760072,3.587490966130752) circle (2.5pt);
\draw[color=sqsqsq] (-2.85,3.65) node {$ \left( 5a + 1,2a + 1 \right) $};
\draw [fill=sqsqsq] (-2.2222809320784105,3.3258784492101796) circle (2.5pt);
\draw[color=sqsqsq] (-2.86,3.35) node {$ \left( 5a,3a \right) $};
\draw [fill=sqsqsq] (-2.637383260400274,3.0073592299159095) circle (2.5pt);
\draw[color=sqsqsq] (-3.72,3.025) node {$ \left( 6a + 3,2a + 5 \right) $};
\draw [fill=sqsqsq] (-3.007359229915903,2.6373832604002816) circle (2.5pt);
\draw[color=sqsqsq] (-4.075,2.65) node {$ \left( 6a + 4,5a + 5 \right) $};
\draw [fill=sqsqsq] (-3.3258784492101725,2.2222809320784216) circle (2.5pt);
\draw[color=sqsqsq] (-3.95,2.23) node {$ \left( 5a,4a \right) $};
\draw [fill=sqsqsq] (-3.587490966130744,1.769154760876023) circle (2.5pt);
\draw[color=sqsqsq] (-4.65,1.775) node {$ \left( 5a + 6,5a + 1 \right) $};
\draw [fill=sqsqsq] (-3.7877205179804148,1.2857578612126697) circle (2.5pt);
\draw[color=sqsqsq] (-4.34,1.285) node {$ \left( 6a,4 \right) $};
\draw [fill=sqsqsq] (-3.923141121612916,0.7803612880645409) circle (2.5pt);
\draw[color=sqsqsq] (-4.92,0.78) node {$ \left( 5a + 2,a + 4 \right) $};
\draw [fill=sqsqsq] (-3.991435692954412,0.26161251692060467) circle (2.5pt);
\draw[color=sqsqsq] (-4.55,0.27) node {$ \left( 4,6a \right) $};
\draw [fill=sqsqsq] (-3.9914356929544166,-0.26161251692053566) circle (2.5pt);
\draw[color=sqsqsq] (-4.62,-0.275) node {$ \left( 4a,5a \right) $};
\draw [fill=sqsqsq] (-3.9231411216129297,-0.7803612880644725) circle (2.5pt);
\draw[color=sqsqsq] (-4.55,-0.775) node {$ \left( 4a,2a \right) $};
\draw [fill=sqsqsq] (-3.7877205179804374,-1.2857578612126033) circle (2.5pt);
\draw[color=sqsqsq] (-4.28,-1.29) node {$ \left( 3,a \right) $};
\draw [fill=sqsqsq] (-3.5874909661307752,-1.7691547608759592) circle (2.5pt);
\draw[color=sqsqsq] (-4.65,-1.765) node {$ \left( 5a + 5,6a + 4 \right) $};
\draw [fill=sqsqsq] (-3.3258784492102134,-2.222280932078361) circle (2.5pt);
\draw[color=sqsqsq] (-4.38,-2.21) node {$ \left( 6a + 4,2a + 2 \right) $};
\draw [fill=sqsqsq] (-3.0073592299159535,-2.6373832604002247) circle (2.5pt);
\draw[color=sqsqsq] (-4.07,-2.63) node {$ \left( 2a + 6,5a + 6 \right) $};
\draw [fill=sqsqsq] (-2.637383260400333,-3.0073592299158594) circle (2.5pt);
\draw[color=sqsqsq] (-3.7,-3.015) node {$ \left( 5a + 2,6a + 3 \right) $};
\draw [fill=sqsqsq] (-2.2222809320784793,-3.3258784492101334) circle (2.5pt);
\draw[color=sqsqsq] (-3.2,-3.3) node {$ \left( 5a + 5,a + 3 \right) $};
\draw [fill=sqsqsq] (-1.769154760876088,-3.587490966130712) circle (2.5pt);
\draw[color=sqsqsq] (-2.8,-3.58) node {$ \left( 2a + 1,2a + 6 \right) $};
\draw [fill=sqsqsq] (-1.2857578612127412,-3.7877205179803903) circle (2.5pt);
\draw[color=sqsqsq] (-2.33,-3.9) node {$ \left( 6a + 3,5a + 2 \right) $};
\draw [fill=sqsqsq] (-0.7803612880646181,-3.923141121612901) circle (2.5pt);
\draw[color=sqsqsq] (-1.15,-4.15) node {$ \left( 0,1 \right) $};
\draw [fill=sqsqsq] (-0.2616125169206878,-3.9914356929544064) circle (2.5pt);
\draw[color=sqsqsq] (-0.4,-4.3) node {$ \left( 0,6 \right) $};

\draw[color=qqqqff] (-0.2,-3.075) node {\fontsize{10}{0} $2$};
\draw[color=qqqqff] (-0.2,-2.2) node {\fontsize{10}{0} $1$};
\draw[color=qqqqff] (-0.5,-0.8) node {\fontsize{10}{0} $ 4 $};
\draw[color=qqqqff] (0.9,2.0) node {\fontsize{10}{0} $1$};
\draw[color=qqqqff] (0.9,-0.9) node {\fontsize{10}{0} $ 4 $};

\draw[color=qqqqff] (-2.85,0.5) node {\fontsize{10}{0} $ 2 $};
\draw[color=qqqqff] (2.95,0.5) node {\fontsize{10}{0} $ 2 $};
\draw[color=qqqqff] (-0.2,3.23) node {\fontsize{10}{0} $ 2 $};

\end{scriptsize}
\end{tikzpicture}

\caption{The points of $ C \left((0,0),1 \right)_{\mathbb{F}_{49}} $ and the squared distances between them and the point $ \left( a + 4,5a + 2 \right) $ marked by red lines for non-rational and blue lines for rational squared distances} 
\label{existence_perf_dist}
\end{center}
\end{figure}
%%%%%%%%%%%%%%%%%%%%%%%%%%%%%%%%%%%%%%%%%%%%%%%%%	

Observe that the circular point sets in \Cref{ex_F7} and \Cref{ex_F7_extended} given by the dark blue lines (where the points have the same squared distances from each other) have maximal cardinality although the circle $ C \left((0,0),1 \right)_{\mathbb{F}_{49}} $ contains more points than $ C \left((0,0),1 \right)_{\mathbb{F}_{7}} $. Hence, it seems that circles over field extensions do not have bigger c-maximal circular point sets than circles over the underlying prime field plane. We will prove this fact later. %, i.e. 
%the new light blue lines do not give us points which have rational squared distances from at least two points  connected with dark blue lines. In this section we would like to answer this question in general. To deal with it (and also to distinguish between light and dark blue lines) we introduce the following notion.

% The last example shows that maximal circular point sets on circles over field extensions are more delicate to treat than over prime fields. It seems that there are two different kind of rational distances: the dark blue ones having the property that they are dividing circular point sets into equivalence classes of points having rational distances from each other and the light blue ones which does not. We introduce the following definition for it.

\begin{definition}
	A rational squared distance between two points of a circle $ C \left(M,r \right)_{\mathbb{F}} $ is called {\em perfect with respect to the underlying circle} if there exists a third point on the same circle such that all the squared distances between these three different points are rational.
\end{definition}
 
In other words a squared distance is called perfect if it is the side of a rational triangle, i.e. all the mutual squared distances of this triangle are rational.
 
%Equivalently, you can also say that a squared distance is perfect if and only if it is a distance in a triangle which sides are of rational length and vertices belonging to the circular point set.
%\newline
Note that the last definition makes sure that a perfect distance is never equal to zero.
In \Cref{ex_F7_extended} the squared distances between the points connected by the dark blue lines are perfect, whereas the same does not hold true for the dashed lines. Indeed, for all pairs of points connected by a light blue line there is no further point on the same circle such that all the squared distances between three of these points are rational. To construct the c-maximal circular point set containing the point $ \left( a + 4,5a + 2 \right) $ it also seems to be enough to know the squared distances $ 2,4 $ and to construct all points having such a squared distance form a starting point on the circle. Thus, these numbers have a special property which will be named in the following. 
% For a squared distance having this property we will give it the following name.

\begin{definition}
Consider a circular point set which contains at least three points. A rational squared distance between two different points of this circle %$ C \left(M,r \right)_{\mathbb{F}} $
has the {\em extension property with respect to the underlying circle} if any point on this circle having this squared distance to one point of the circular point set implies that this point has r.s.d. to all points of this circular point set.
\end{definition}

%Note that both properties have geometric nature 
Note that the last two definitions are purely geometric and the extension property of a squared distance implies that it is also perfect. However, it is not obvious that the inverse also holds. We will need the following proposition and lemma which will also bring an algebraic property of perfect distances into play:

\begin{proposition} \label{prop_algebraic_con}
	Let $ q $ be a perfect distance with respect to $ C \left(r,M \right)_{\mathbb{F}} $. Then the following is satisfied: 
\begin{equation*}
q \in \square_{\boldsymbol{P} \left( \mathbb{F} \right)}  \quad \textrm{and} \quad 1 - \tfrac{q}{4r^2} \in \square_{\boldsymbol{P} \left( \mathbb{F} \right)}  
\tag{*}
% \qquad \qquad (*)
\end{equation*}
\end{proposition}
	
\begin{proof}
	$ q \in \square_{\boldsymbol{P} \left( \mathbb{F} \right)} $ is satisfied because $ q $ is rational. Moreover, since $ q $ is perfect, we find three points in $ C \left(M,r \right)_{\mathbb{F}} $ such that all the squared distances between the points are rational and one of them is equal to $ q $. Because squared distances are invariant under translation and rotation, we can assume that $ M =\left( 0,0 \right) $ and the point opposite of the side with squared distance $ q $ is $ P_0 \coloneqq \left( 0,r \right) $. Assume that the other points $ P_1,P_2 $ are parametrized by $ t_1,t_2 \in \mathbb{F} $ and the squared distances between $ P_0 $ and $ P_i $ is $ q_i $ for $ i = 1,2 $. Applying \Cref{dist_lemma} we have
	\begin{align*}
		q_i = \frac{4r^2}{t_i^2 + 1} \in \square_{\boldsymbol{P} \left( \mathbb{F} \right)} & & \mathrm{and \ so} & & q = \dfrac{4r^2 \left( t_1 - t_2 \right)^2}{ \left( t_1^2 + 1 \right) \left( t_2^2 + 1\right)} = \frac{q_1 q_2}{4r^2} \left( t_1 - t_2 \right)^2 \in \square_{\boldsymbol{P} \left( \mathbb{F} \right)}.
	\end{align*}
Since $ \square_{\boldsymbol{P} \left( \mathbb{F} \right)^{*}} $ is a group with respect to multiplication and $ q_i \neq 0 $, we can deduce
	\begin{align*}
	\frac{1}{4r^2} \left( t_1 - t_2 \right)^2 = \left( \frac{1}{2r} \left( t_1 -t_2 \right) \right)^2 \in \square_{\boldsymbol{P} \left( \mathbb{F} \right)}.
	\end{align*}
Hence, we get	
%Note that the last term is a square in a field and therefore we have
%for it having different parity. Hence, we get that both of these square roots are in the underlying prime field which gives us that
	\begin{align*}
		\frac{1}{2r} \left( t_1 -t_2 \right) \in {\boldsymbol{P} \left( \mathbb{F} \right)}.
	\end{align*}
	
On the other hand, we have
	\begin{align*}
		\left( t_1 + t_2 \right) \left( t_1 - t_2 \right) = \left( t_1^2+1 \right) - \left(t_2^2+1 \right) = \frac{4r^2}{q_1} - \frac{4r^2}{q_2} = 4r^2 \frac{q_2 - q_1}{q_1 q_2}.
	\end{align*}
By the above results and the fact that the $ t_i $'s are different because they are parametrizing different points we get
	\begin{align*}
		\frac{1}{2r} \left( t_1 + t_2 \right) = \underbrace{\left( \frac{1}{2r} \left( t_1 -t_2 \right) \right)^{-1}}_{\in \boldsymbol{P} \left( \mathbb{F} \right)} \underbrace{ \frac{q_2 - q_1}{q_1 q_2}}_{\in \boldsymbol{P} \left( \mathbb{F} \right)} \in {\boldsymbol{P} \left( \mathbb{F} \right)}
	\end{align*}
so we can deduce $ \frac{t_i}{2r} \in \boldsymbol{P} \left( \mathbb{F} \right) $ and also $ \frac{t_i^2}{4r^2} \in \square_{\boldsymbol{P} \left( \mathbb{F} \right)} $. Thus, we conclude
	\begin{align*}
		1 - \frac{q_i}{4r^2} = \frac{q_i}{4r^2} \left( \frac{4r^2}{q_i} -1 \right) = \frac{q_i}{4r^2} \left( {t_i}^2 + 1 -1 \right) = \underbrace{q_i}_{\in \square_{\boldsymbol{P} \left( \mathbb{F} \right)}} \cdot \underbrace{\frac{t_i^2}{4r^2}}_{\in \square_{\boldsymbol{P} \left( \mathbb{F} \right)}} \in \square_{\boldsymbol{P} \left( \mathbb{F} \right)}
	\end{align*}
for $ i = 1,2 $.	
\end{proof}	

\begin{definition}	
We also say that a squared distance $ q $ satisfies the {\em algebraic circle property (short a.c.p.) with respect to $ C \left(M,r \right)_{\mathbb{F}} $} if $ q $ satisfies $ ({}^*) $. % for a squared distance $  q $.
	\end{definition}

A direct consequence of \Cref{prop_algebraic_con} and a necessary condition for the existence of a circular point set with cardinality at least three on $ C \left(M,r \right)_{\mathbb{F}} $ is the following.

\begin{corollary} \label{coro_radius}
	If we have a circular point set contained in $ C \left(M,r \right)_{\mathbb{F}} $ with cardinality at least three, then $ r^2 \in \boldsymbol{P} \left( \mathbb{F} \right) $.
	%of point in $ C \left(r,M \right)_{\mathbb{F}} $, then we have $ r^2 \in \boldsymbol{P} \left( \mathbb{F} \right) $.
\end{corollary}

\begin{proof}
	Assume that two different points from $ C \left(M,r \right)_{\mathbb{F}} $ have squared distance $ q $. Since $ C \left(M,r \right)_{\mathbb{F}} $	is a circular point set, all squared distances between the points in $ C \left(M,r \right)_{\mathbb{F}} $ are rational and so perfect because it contains at least three points. By \Cref{prop_algebraic_con} we get that $ q \in \square_{\boldsymbol{P} \left( \mathbb{F} \right)^{*}} \subseteq \boldsymbol{P} \left( \mathbb{F} \right)^{*} $ and $ 1 - \tfrac{q}{4r^2} \in \square_{\boldsymbol{P} \left( \mathbb{F} \right)} \subseteq \boldsymbol{P} \left( \mathbb{F} \right) $. Since $ \boldsymbol{P} \left( \mathbb{F} \right) $ is a field and $ 4 \in \boldsymbol{P} \left( \mathbb{F} \right) $ we conclude that $ r^2 \in \boldsymbol{P} \left( \mathbb{F} \right) $.
\end{proof}

The next part is the last step to show that perfect distances also have the extension property.

\begin{lemma} \label{lem_perfect_means_propagation}
	Assume that there is a circular point set on $ C \left((0,0),r \right)_{\mathbb{F}} $ with cardinality at least $ 3 $ and let $ q \neq 0 $ be a rational squared distance with $ 1 - \frac{q}{4r^2} \in \square_{\boldsymbol{P} \left( \mathbb{F} \right)} $. Then $ q $ has the extension property with respect to $ C \left((0,0),r \right)_{\mathbb{F}} $.
\end{lemma}

\begin{proof}
	Without loss of generality, we can assume that we have a point $ P =  \left( 0, r\right) $ which belongs to a circular point set on $ C \left((0,0),r \right)_{\mathbb{F}} $ of cardinality at least three and a point $ Q \in C \left((0,0),r \right)_{\mathbb{F}} $ such that the squared distance between $ P $ and $ Q $ is $ q $. We have to show now that $ Q $ and an arbitrary point $ R $ contained in the circular point set have r.s.d.. 
%we can assume that our circular point set is on $ C \left(r,(0,0) \right)_{\mathbb{F}} $ and contains the point $ P = \left( 0,-r \right) $, where $ q $ is the squared distance between $ P $ and a point $ Q $ on the same circle. 
Let $ R $ be parametrized by the parameter $ t_R $ and let $ q' $ be the rational squared distance between $ P $ and $ R $. Note that $ q' $ satisfies the algebraic circle property by \Cref{prop_algebraic_con}. Let $ Q $ be parametrized by $ t_Q $. 
% It is enough to show that $ R $ and $ Q $ have rational squared distance from each other. 
By \Cref{dist_lemma} we get
	\begin{align*}
		D^2 \left( Q, R \right) = \frac{q q'}{4r^2} \left( t_R - t_Q \right)^2 = q q' \left( \frac{t_R}{2r} - \frac{t_Q}{2r} \right)^2
	\end{align*}
as in the proof of \Cref{prop_algebraic_con}. Now we need to explain, why the last term in brackets is contained in the prime field. Since $ q' $ must be perfect, both $ q $ and $ q' $ satisfy the algebraic condition $ ({}^*) $ in \Cref{prop_algebraic_con}. Then we also have 
	\begin{align*}
		\frac{q}{4r^2} t_Q^2 = 1- \frac{q}{4r^2} \in \square_{\boldsymbol{P} \left( \mathbb{F} \right)} & & \mathrm{and} & & \frac{q'}{4r^2} t_R^2 = 1- \frac{q'}{4r^2} \in \square_{\boldsymbol{P} \left( \mathbb{F} \right)}
	\end{align*}	
which implies 
	\begin{align*}
		\frac{t_Q^2}{4r^2} \in \square_{\boldsymbol{P} \left( \mathbb{F} \right)} & & \mathrm{and} & & \frac{t_R^2}{4r^2} \in \square_{\boldsymbol{P} \left( \mathbb{F} \right)} 		
	\end{align*}
since % $ r^2 \in \boldsymbol{P} \left( \mathbb{F} \right) $ because of \Cref{coro_radius} and 
$ q,q' \in \square_{\boldsymbol{P} \left( \mathbb{F} \right)} $.	 Thus we have $ D^2 \left( Q, R \right) \in \square_{\boldsymbol{P} \left( \mathbb{F} \right)} $. %which let us conclude.
\end{proof}	

\begin{example} \label{ex_non_perf}
	Consider the circle $ C \left((0,0),1 \right)_{\mathbb{F}_5} $, then 
	 $$ C \left((0,0),1 \right)_{\mathbb{F}_5} = \left\{ \left(1,0 \right), \left( 4,0\right), \left(0,1 \right), \left(4,1 \right) \right\} $$
and the e-maximal circular point sets on $ C \left((0,0),1 \right)_{\mathbb{F}_5} $ have cardinality equal to $ 2 $ by \Cref{theorem_prime_field}. This means there is neither a perfect squared distance nor a squared distance with propagation property between the points of $ C \left((0,0),1 \right)_{\mathbb{F}_5} $. Although we have that 
	$$ q \coloneqq D^2 \left( \left(1 , 0 \right), \left( 4, 0\right) \right) = 4 \in \square_{\mathbb{F}_5} $$
and 
	$$ 1 - \frac{q}{4r^2} = 1 - \frac{4}{4} = 0 \in \square_{\mathbb{F}_5}, $$
so $ q $ satisfies $ ({}^*) $. In case we would not require for the extension property that the underlying circle must contain a circular point set of cardinality at least three, then $ 4 $ would have the extension property without being perfect and so there would not be an equivalence to these two terms.
\end{example}

	\begin{example}
We would like to calculate all the possible perfect distances with respect to $ C \left((0,0),1 \right)_{\mathbb{F}_7} $ and $ C \left((0,0),1 \right)_{\mathbb{F}_{49}} $. Clearly perfect distances must be elements of the set $ \square_{\mathbb{F}_7^{*}} = \left\{ 1,2,4 \right\} $. Moreover, by considering \Cref{F_7_with_distances} we see that $ 2 $ and $ 4 $ must be perfect distances by definition. Indeed, they also satisfy
	\begin{align*}
		1 - \frac{2}{4} = 1 - 4 = 4 \in \square_{\mathbb{F}_7} \\
		1 - \frac{4}{4} = 1-1 = 0 \in \square_{\mathbb{F}_7}
	\end{align*}
and so $ ({}^*) $ from \Cref{prop_algebraic_con} holds. Whereas
	$$ 1 - \frac{1}{4} = 1-2 = 6 \notin \square_{\mathbb{F}_7} $$
is not a perfect distances because $ ({}^*) $ is violated. This can also be seen in \Cref{ex_F7_extended} for $ C \left((0,0),1 \right)_{\mathbb{F}_{49}} $ as there does not exist three points with rational squared distances from each other such that one of the squared distances is equal to $ 1 $. So surprisingly all rational squared distances occurring in $ C \left( \left( 0,0 \right),1 \right)_{\mathbb{F}_7} $ are perfect. We will discuss this more generally in the next lemma.
\end{example}

%perfect distances with respect to the circle $ C \left( \left( 0,0 \right),1 \right)_{\mathbb{F}_7} $ already occurs there and the extension of the field $ \mathbb{F}_7 $ to $ \mathbb{F}_{49} $ does not increase the number of perfect distances. We will discuss this more generally in the next lemma.

%Surprisingly, all the occurring rational squared distances between points on circles over prime field planes are perfect distances and it also holds the other way round, i.e. for all perfect distances you find two points on a circle over a prime field plane such that the squared distances of these point is equal to a specific perfect distance.

%\Cref{perf_vers_occurring} means that the blue colored distances in \Cref{F_7_with_distances} are the only perfect distances you can find with respect to the circle  $ C \left( \left( 0,0 \right),1 \right)_{\mathbb{F}_7} $.
	
\begin{lemma} \label{perf_vers_occurring}
	Let $ \mathbb{F} $ be a prime field, $ M \in \mathbb{F}^2 $, $ r \in \mathbb{F}^* $ and $ q \in \square_{ \mathbb{F}^*} $, then $ q $ is perfect with respect to a circle $ C \left( M,r \right)_{\mathbb{F}} $ if $ q \neq 4r^2 $ and there exist two points $ P_1,P_2 \in C \left( M,r \right)_{\mathbb{F}} $ such that $ D^2 \left( P_1, P_2 \right) = q $.
\end{lemma}

\begin{proof}
	Without loss of generality and by the uniformity property we can assume that $ M = \left( 0,0 \right) $ and $ P_1 = \left( r,0 \right) $. Assume further that $ D^2 \left( P_1, P_2 \right) = q $ and set $ P_2 \coloneqq  \left( x,y \right) \in C \left( M,r \right)_{\mathbb{F}} $. Then the following equations are satisfied:
	\begin{align*}
		\left( x-r \right)^2 + y^2 &= q \\
		x^2 + y^2 &= r^2
	\end{align*}
If we solve for $ x $ we get $ x = r - \tfrac{q}{2r} $. Solving for $ y^2 $ we get $ y^2 = q \left( 1 - \tfrac{q}{4r^2} \right) \neq 0 $. Now for any choice of $ y =  \sqrt{q \left( 1 - \tfrac{q}{4r^2} \right)} $ we can consider the point $ P_3 \coloneqq \left( x,-y \right) $ and we also have $ D^2 \left( P_1, P_3 \right) = q $. Since $ D^2 \left( P_2, P_3 \right) = 4q \in \square_{\mathbb{F}} $ we conclude that $ q $ is perfect.
\end{proof}

%and so the occurrence of a
% rational squared distance $ q $ in $ C \left( M,r \right)_{\mathbb{F}} $ is equivalent of $ P_2 = \left( x,y\right) \in C \left( M,r \right)_{\mathbb{F}} $ and this is equivalent to $ 1 - \tfrac{q}{4r^2} \in \square_{\mathbb{F}} $ which is equivalent to $ q $ being perfect by \Cref{prop_algebraic_con}.

\begin{example} 
We will now consider an application of perfect distances. Assume we have an arbitrary prime field $ \mathbb{F} $, $ M \in \mathbb{F}^2 $, $ r \in \mathbb{F} $ and three different points $ A,B,C \in C \left( M,r \right)_{\mathbb{F}} $ defining a triangle such that the squared distances $ a $ and $ b $ are rational (see \Cref{existence_perf_dist}). Then all squared distances between the points $ A,B,C $  are rational by \Cref{perf_vers_occurring} as an immediate consequence.

	% Assume now we have a triangle with rational vertices in the plane circumscribed by a circle with a rational radius (i.e.). If two sides of the triangle are rational, then all sides are rational. For the proof of it we can set $ \mathbb{F} = \mathbb{Q} $. Now since the vertices are rational, we also have that the center of the circumscribing circle $ M \in \mathbb{Q} $ and hence the three vertices are contained in $ C \left( \left(M,r \right) \right)_{\mathbb{Q}} $. Since $ \mathbb{Q} $ is a prime field, we can apply \Cref{perf_vers_occurring} and we get that the rational sides of the triangle are in fact perfect distances.
	
%Implication:	
%	 By \Cref{lem_perfect_means_propagation} perfect distances also have the propagation property, and so the length of third side must also be a square in $ \mathbb{Q} $, i.e. rational.	
	
The case $ \mathbb{F} = \mathbb{Q} $ this could also be proved without any knowledge of perfect distances and extension property in the following way: By the law of sines we have
	%$$\frac{a}{\sin\left( \alpha \right)} = \frac{b}{\sin\left( \beta \right)} = \frac{c}{\sin\left( \gamma \right)} = 2r \in \mathbb{Q} $$
	$$ \frac{\sqrt{c}}{\sin\left( \gamma \right)} = 2r \in \mathbb{Q} $$
since the three points are circumscribed by a circle with rational radius $ r $. %If we assume that $ a,b,c $ are the lengths of the triangles and $ a,b \in \mathbb{Q} $, then we clearly get that $ \sin\left( \alpha \right), \sin\left( \beta \right) \in \mathbb{Q} $. 
Moreover, we can calculate the area of a parallelogram defined by the connection vectors of $ CB $ and $ CA $, denoted by $ \vec{a}, \vec{b} \in \mathbb{Q}^2 $, by the formula
	$$ \sin \left( \gamma \right) \sqrt{a} \sqrt{b} = \vec{a} \times \vec{b} \in \mathbb{Q} $$
which means that $ \sin \left( \gamma \right) \in \mathbb{Q} $ and therefore also $ \sqrt{c} \in \mathbb{Q} $.

Note in case that $ \mathbb{F} $ is any prime field, then we can still prove the generalized statement above even if we do not have tools like sinus.

%%%%%%%%%%%%%%%%%%%%%%%%%%%%%%%%%%%%%%%%%%%%%%%%%	
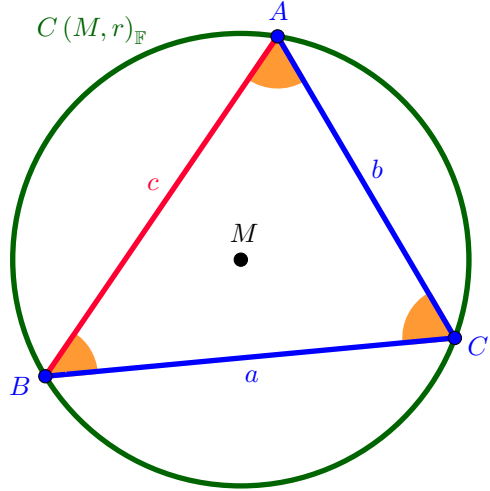
\begin{figure}[h]  
\begin{center}
\usetikzlibrary{arrows}
\pagestyle{empty}
\newcommand{\degre}{\ensuremath{^\circ}}
\definecolor{ffzztt}{rgb}{1.,0.6,0.2}
\definecolor{ffqqtt}{rgb}{1.,0.,0.2}
\definecolor{qqqqff}{rgb}{0.,0.,1.}
\definecolor{qqwuqq}{rgb}{0.,0.39215686274509803,0.}
\begin{tikzpicture}[line cap=round,line join=round,>=triangle 45,x=1.0cm,y=1.0cm]
\clip(0.8265174099083505,0.19456270551648566) rectangle (9.254143606615882,7.834151127986416);
\draw [shift={(2.429295003558384,2.4535602755780905)},line width=2.pt,color=ffzztt,fill=ffzztt,fill opacity=0.10000000149011612] (0,0) -- (5.392924982272518:0.6566981451979895) arc (5.392924982272518:55.87846816637941:0.6566981451979895) -- cycle;
\draw [shift={(5.483395696706457,6.960798642327045)},line width=2.pt,color=ffzztt,fill=ffzztt,fill opacity=0.10000000149011612] (0,0) -- (-124.12153183362061:0.6566981451979895) arc (-124.12153183362061:-59.758144961307195:0.6566981451979895) -- cycle;
\draw [shift={(7.814688378199527,2.9619588959831527)},line width=2.pt,color=ffzztt,fill=ffzztt,fill opacity=0.10000000149011612] (0,0) -- (120.24185503869283:0.6566981451979895) arc (120.24185503869283:185.39292498227252:0.6566981451979895) -- cycle;
\draw [line width=2.pt,color=qqwuqq] (5.,4.) circle (3.cm);
\draw [line width=2.pt,color=qqqqff] (7.814688378199527,2.9619588959831527)-- (5.483395696706457,6.960798642327045);
\draw [line width=2.pt,color=qqqqff] (7.814688378199527,2.9619588959831527)-- (2.429295003558384,2.4535602755780905);
\draw [line width=2.pt,color=ffqqtt] (2.429295003558384,2.4535602755780905)-- (5.483395696706457,6.960798642327045);
\begin{scriptsize}
\draw [fill=black] (5.,4.) circle (2.5pt);
\draw[color=black] (5.0,4.35) node {\fontsize{10}{0} $M$};
\draw[color=qqwuqq] (3.0, 7.05) node {\fontsize{10}{0} $ C \left( M,r \right)_{\mathbb{F}} $};
\draw [fill=qqqqff] (7.814688378199527,2.9619588959831527) circle (2.5pt);
\draw[color=qqqqff] (8.075,2.85) node {\fontsize{10}{0} $C$};
\draw [fill=qqqqff] (5.483395696706457,6.960798642327045) circle (2.5pt);
\draw[color=qqqqff] (5.45,7.3) node {\fontsize{10}{0} $A$};
\draw [fill=qqqqff] (2.429295003558384,2.4535602755780905) circle (2.5pt);
\draw[color=qqqqff] (2.05,2.325) node {\fontsize{10}{0} $B$};
\draw[color=qqqqff] (6.75,5.2) node {\fontsize{10}{0} $b$};
\draw[color=qqqqff] (5.1,2.45) node {\fontsize{10}{0} $a$};
\draw[color=ffqqtt] (3.8,5.0) node {\fontsize{10}{0} $c$};
\draw[color=ffzztt] (2.75,2.65) node {\fontsize{10}{0} $\beta $};
\draw[color=ffzztt] (5.45,6.525) node {\fontsize{10}{0} $\alpha $};
\draw[color=ffzztt] (7.4,3.1) node {\fontsize{10}{0} $\gamma $};
\end{scriptsize}
\end{tikzpicture}

\caption{Three points $ A,B,C $ on a circle $ C \left(M,r \right)_{\mathbb{F}} $ over a prime field plane $ \mathbb{F} $ with rational squared distances denoted by $ a, b $ and squared distance $ c $.} 
\label{existence_perf_dist}
\end{center}
\end{figure}
%%%%%%%%%%%%%%%%%%%%%%%%%%%%%%%%%%%%%%%%%%%%%%%%%	
\end{example}

\end{subsection}

\begin{subsection}{C-maximal circular point sets on circles over arbitrary field planes}

The goal of this section is to determine the cardinality of c-maximal circular point sets % also with respect to the cardinality 
of a circle $ C \left(M,r \right)_{\mathbb{F}} $ over an arbitrary field plane $ \mathbb{F} \times \mathbb{F} $ where $ M \in \mathbb{F}^2 $ and $ r \in \mathbb{F}^{*} $. We will assume that $ \chr \left( \mathbb{F} \right) \neq 2 $. %From the last section we can use the algebraic circle property to determine whether a distance is perfect or not. %However, if the prime field is large or infinite, then it could time-consuming to test all of them and we are interested in knowing the cardinality of it without considering every single case separately. (Observe also that a distance $ q $, which is perfect with respect to a radius $ r $ gives us two new points on the circle with rational distances from each other if and only if $ q \neq 4r^2 $. Otherwise only one.)
Now we would like to find a parametrization of perfect distances since then we can construct them directly.
%For this we would like to find a parametrization of perfect distances such that we can them calculate directly. Indeed, this is possible as we can see in the following.

\begin{proposition} \label{prop_alg_crit}
	A squared distance $ q \neq 4r^2 $ is perfect with respect to $ C \left( M,r \right)_{\mathbb{F}} $ if and only if there is $ t \in \boldsymbol{P} \left( \mathbb{F} \right)^{*} $ such that $ q = \left( \tfrac{4tr^2}{t^2 + r^2} \right)^2 $.
\end{proposition}

\begin{proof}
	Let $ q $ be perfect, then the algebraic property $ ({}^*) $ holds true for $ q $ by \Cref{prop_algebraic_con}. Hence, we can find a solution $ \left( x_0, y_0 \right) \in \boldsymbol{P} \left( \mathbb{F} \right)^2 $ which satisfy the equation of the ellipse
%Hence, there are either four or two corresponding solutions to the ellipse 
	\begin{displaymath}
		\epsilon: x^2 + 4r^2y^2 = 4r^2
	\end{displaymath}
%i.e. we can find $ \left( \pm x, \pm y \right) \in \boldsymbol{P} \left( \mathbb{F} \right) \times \boldsymbol{P} \left( \mathbb{F} \right) $ (which are only two possibilities if either $ x = 0 $ or $ y = 0 $)
such that $ x_0^2 = q $ and $ y_0^2 = 1- \tfrac{q}{4r^2} $. The goal is now to find $ t_0 \in \boldsymbol{P} \left( \mathbb{F} \right) $ such that $ x_0 $ can be parametrized as above. Through the points $ \left( x_0, y_0 \right) $ and $ \left( 0,-1 \right) $ we find a line which is defined by the equation 
	$$ l: y = m_0x -1 $$
for some $ m_0 \in \boldsymbol{P} \left( \mathbb{F} \right)^{*} $. Therefore $ x_0,y_0,m_0 $ satisfy the following equation
	\begin{displaymath}
		x_0^2 + 4r^2 \left( m_0x_0-1 \right)^2 = 4r^2.
	\end{displaymath}
As $ x_0 \neq 0 $ we can divide by $ x_0 $ and solve for it. By defining $ t_0 \coloneqq \tfrac{1}{2m_0} $ we finally get
	\begin{align*}
		x_0 = \frac{8 m_0 r^2}{1 + 4 m_0^2 r^2} = \frac{4t_0r^2}{t_0^2 + r^2}.
	\end{align*}

For the other direction, let $ q = \left( \tfrac{4tr^2}{t^2 + r^2} \right)^2 $ for $ t \in \boldsymbol{P} \left( \mathbb{F} \right)^{*} $. At first we will show that the algebraic circle property is satisfied by $ q $. Since $ t \in \boldsymbol{P} \left( \mathbb{F} \right)^{*} $ and $ r^2 \in \boldsymbol{P} \left( \mathbb{F} \right) $ by \Cref{coro_radius}, we get $ q \in \square_{\boldsymbol{P} \left( \mathbb{F} \right)} $. Moreover, we have 
	\begin{align*}
		1 - \frac{q}{4r^2} = \frac{ \left( t^2 + r^2 \right)^2 }{ \left( t^2 + r^2 \right)^2} - \frac{4t^2r^2}{\left( t^2 + r^2 \right)^2} = \bigg( \underbrace{ \frac{ t^2 - r^2 }{ t^2 + r^2 } }_ {\in \boldsymbol{P} \left( \mathbb{F} \right)} \bigg)^2 \in \square_{\boldsymbol{P} \left( \mathbb{F} \right)}.
	\end{align*}
Note that $ 1 - \tfrac{q}{4r^2} $ is not equal to zero because $ q \neq 4r^2 $. We would like to show now that there exist three points in $ C \left( M,r \right)_{\mathbb{F}} $ such that all the mutual distances are rational and at least one of the sides has squared distance $ q $. For this set $ x \coloneqq r - \tfrac{q}{2r} \in \mathbb{F} $ and $ y = \sqrt{q \left( 1 - \tfrac{q}{4r^2} \right)} \in \mathbb{F} $ where the sign does not matter. We can define $ P_1 \coloneqq \left(r,0 \right) $, $ P_2 \coloneqq \left(x,y \right) $ and $ P_3 \coloneqq \left(x,-y\right) $ as in the proof of \Cref{perf_vers_occurring} and show that the mutual distances between the points are all rational and the rational distance $ q $ occurs. Hence, $ q $ is perfect.	
\end{proof}

The last step will be to generalize \Cref{theorem_prime_field} for the case of an arbitrary field $ \mathbb{F} $. But before we can prove it we need the following.

	\begin{lemma} \label{lemma_two_points}
Let $ P \in C \left(M,r \right)_{\mathbb{F}} $ and let $ q $ be a perfect distance with respect to $ C \left(M,r \right)_{\mathbb{F}} $. Then there exists $ Q \in C \left(M,r \right)_{\mathbb{F}} $ such that $ D^2 \left( P,Q \right) = q $. Moreover, there exists two such points $ Q_1 \neq Q_2 $ if and only if $ q \neq 4r^2 $.
	\end{lemma}

\begin{proof}
	It is enough to show the statement for $ P = \left( 0,-r \right) $ and $ M = \left( 0,0 \right) $. Since $ q $ is perfect, there exist roots $ \alpha, \beta \in \boldsymbol{P} \left( \mathbb{F} \right)^{*} $ of $ q $ and $ 1 - \frac{q}{4r^2} $, respectively. Let $ q_1 \coloneqq \alpha \beta $ and $ q_2 \coloneqq \tfrac{q}{2r} -r $. Define $ Q^{+} \coloneqq \left( q_1, q_2 \right) $. Then we have
	\begin{align*}
		q_1^2 + q_2^2 = q \left( \frac{4 r^2 - q}{4r^2} \right) + \left( \frac{q^2 - 4qr^2 + 4r^4}{4r^2} \right) = r^2,
	\end{align*}
which implies $ Q^{+} \in C \left(r,M \right)_{\mathbb{F}} $. Moreover, we have
	\begin{align*}
		D^2 \left( P,Q^{+} \right) = q_1^2 + \left( q_2 + r \right)^2 = q \left( \frac{4 r^2 - q}{4r^2} \right) + \frac{q^2}{4r^2} = q.
	\end{align*}
Now, let us define $ Q^{-} \coloneqq \left( -q_1, q_2 \right) $, then it is easy to see that we also have $ Q^{-} \in C \left(M,r \right)_{\mathbb{F}} $ and that the above equations can have at most one solution for $ q_2 $ (observe that $ q_2 $ is uniquely defined by them).
Hence, we found all points satisfying the equations from this lemma and $ Q_1 \neq Q_2 $, i.e. $ \alpha \beta \neq - \alpha \beta $. Since the characteristic of $ \mathbb{F} $ is different from two and $ q \neq 0 $ this is equivalent to $ q \neq 4r^2 $.
\end{proof}

%\clearpage

\begin{theorem} \label{theo_general_c-maximal}
	Let $ C \left(M,r \right)_{\mathbb{F}} $ be an arbitrary circle over an arbitrary field $ \mathbb{F} $. Then the c-maximal circular point sets have the following cardinalities:
	\begin{table}[h]
\begin{tabularx}{\textwidth}{p{0.22\textwidth}  |X|c|c|c}
characteristic & cases & $ r \in \boldsymbol{P} \left( \mathbb{F} \right) $ & $ r^2 \in \boldsymbol{P} \left( \mathbb{F} \right) \not\ni r $ & $ r^2 \notin \boldsymbol{P} \left( \mathbb{F} \right) $ \\
\hline
$ \chr \left( \mathbb{F} \right) = 2 $ & &  $ \lvert \mathbb{F} \rvert $ &  $ \lvert \mathbb{F} \rvert $ &  $ \lvert \mathbb{F} \rvert $ \\
\hline
$ \chr \left( \mathbb{F} \right) = 3 $ & &  $ 2 $ &  $ 2 $ &  $ \leq 2 $ \\
\hline
%\rule{3mm}{.pt}
\multirow{2}{*}{ $ 3  <  \chr \left( \mathbb{F} \right)  <  \infty $}
	& $ \nexists t \in \boldsymbol{P} \left( \mathbb{F} \right): t^2 + 1 = 0 $ & $ \frac{\lvert \chr \left( \mathbb{F} \right) \rvert + 1}{2} $ & $ \frac{\lvert \chr \left( \mathbb{F} \right) \rvert - 1}{2} $ & $ \leq 2 $ \\
\cline{2-5}
	& $ \exists t \in \boldsymbol{P} \left( \mathbb{F} \right): t^2 + 1 = 0 $ & $ \frac{\lvert \chr \left( \mathbb{F} \right) \rvert - 1}{2} $ & $ \frac{\lvert \chr \left( \mathbb{F} \right) \rvert + 1}{2} $ & $ \leq 2 $ \\
\hline	
 $ \chr \left( \mathbb{F} \right) = \infty $ & & $ \aleph_0 $ & $ \aleph_0 $ & $ \leq 2 $
\end{tabularx}
\caption{Cardinalities of c-maximal circular point sets over arbitrary field planes}
\label{table_for_theo}
\end{table}	
\end{theorem}

\begin{proof}	
	Fifteen different cases need to be proven. If $ \chr \left( \mathbb{F} \right) = 2 $, then the result follows by \Cref{dist_lemma}. In the following we assume that $ \chr \left( \mathbb{F} \right) \neq 2 $. At first we show that there are no perfect distances between points of a circle of radius $ r $ if $ r^2 \notin \boldsymbol{P} \left( \mathbb{F} \right) $. This is a direct consequence of \Cref{coro_radius}.
%because for a perfect distance we find $ t \in \boldsymbol{P} \left( \mathbb{F} \right)^{*} $ such that 
%	$$ q = \left( \tfrac{4tr^2}{t^2 + r^2} \right)^2 \in \square_{\boldsymbol{P} \left( \mathbb{F} \right)^{*}} $$
%by 
%\Cref{prop_algebraic_con} and 
%\Cref{prop_alg_crit}. Hence, there is $ q' \in \boldsymbol{P} \left( \mathbb{F} \right) $ such that $ q' = \tfrac{4tr^2}{t^2 + r^2} $ and so $ r^2 = \tfrac{q't^2}{4t-q'} \in \boldsymbol{P} \left( \mathbb{F} \right) $ which is a contradiction to our assumption. In case there exist no perfect distance with respect to a given circle of radius $ r $, then there are no three points on this circle which pairwise rational squared distance from each other. 
Thus, we can go on with the cases where $ r^2 \in \boldsymbol{P} \left( \mathbb{F} \right) $ and $ 2 <  \chr \left( \mathbb{F} \right) < \infty $.
	
Let 	$ \chr \left( \mathbb{F} \right) = 3 $. We start with the case $ r \in \boldsymbol{P} \left( \mathbb{F} \right) = \mathbb{F}_3 $. Since we always assume that $ r \neq 0 $, we have that either $ r = 1 $ or $ r = 2 $ in $ \mathbb{F}_3 $. Hence, $ r^2 = 1 $. Since $ t \in \boldsymbol{P} \left( \mathbb{F} \right)^{*} $ we have the only squared distance $ q = \left( \tfrac{4tr^2}{t^2 + r^2} \right)^2 = 1 $ with respect to any circle $ C \left(M,r \right)_{\mathbb{F}} $ for $ M \in \mathbb{F} $ arbitrary. However, we also have $ 4r^2 = 1 $ and so for each point on $ C \left(M,r \right)_{\mathbb{F}} $ there exist another unique point on $ C \left(M,r \right)_{\mathbb{F}} $ such that the squared distances between them is equal to $ 1 $ by \Cref{lemma_two_points}.
%the perfect distance $ 1 $ appears exactly one time by \Cref{lemma_two_points} $ C \left(M,r \right)_{\mathbb{F}} $ there exists a unique point such that the squared distance between them is equal to $ 1 $.
Since there are no more perfect distances we deduce that there are no three different points on $ C \left(M,r \right)_{\mathbb{F}} $ such that the pairwise squared distances are rational.
\newline
Consider the case $ r^2 \in \boldsymbol{P} \left( \mathbb{F} \right) $, but $ r \notin \boldsymbol{P} \left( \mathbb{F} \right) $. Then we deduce $ r^2 = 2 $ as $ \boldsymbol{P} \left( \mathbb{F} \right) \setminus \square_{\boldsymbol{P} \left( \mathbb{F} \right)} = \left\{ 2 \right\} $. Since $ t \in \boldsymbol{P} \left( \mathbb{F} \right)^{*} $ for the parametrization of a perfect distance we conclude that there is no perfect distance as $ t^2 = 1 $ and then we would have a division by zero by \Cref{prop_alg_crit}. Hence, it is clear that the c-maximal circular point sets in $ \mathbb{F}^2 $ have cardinality at most two. We will show now that they really have cardinality equal to $ 2 $. Consider the points $ \left( r,0 \right) \in \mathbb{F}^2 $  and $ \left( r - \tfrac{1}{2r},r \right) \in \mathbb{F}^2 $. The squared distance of them is
	$$ \frac{1}{4r^2} + r^2 = \frac{1}{2} + 2 = 2+2 = 1 \in \square_{\boldsymbol{P} \left( \mathbb{F} \right)} .$$ Hence, the c-maximal circular point sets are of cardinality $ 2 $.
	
We consider now the cases where $ 3 <  \chr \left( \mathbb{F} \right) < \infty $. Define 
	$$ q_r(t) = \left( \frac{4tr^2}{t^2 + r^2} \right)^2 $$
for any $ t \in \boldsymbol{P} \left( \mathbb{F} \right)^{*} $ such that $ t $ is admissible, i.e. $ t^2 \neq -r^2 $. We would like to count all the prefect distances we can construct. For the following part we will assume that all rational squared distances equal to $ 4r^2 $ are perfect and discuss the special cases where this is not true below. Observe that for all admissible $ t $ we have
	$$ q_r\left(\pm \frac{t}{r^2}\right) = \Bigg( \frac{4\tfrac{r^2}{t}r^2}{\bigl(\tfrac{r^2}{t}\bigr)^2 + r^2} \Bigg)^2 = \left( \frac{4tr^2}{r^2 + t^2} \right)^2 = q_r \left( \pm t \right), $$
so $ q_r $ gives us the same perfect distance when evaluated on $ \pm t, \pm \tfrac{r^2}{t} \in \boldsymbol{P} \left( \mathbb{F} \right)^{*} $. Moreover, there are no other values in $ \boldsymbol{P} \left( \mathbb{F} \right)^{*} $ such that we get the same perfect distance as for the other four values above because the highest power of $ t $ in $ q_r $ is $ 4 $. However, it might happen that not all of the above four values are different from each other. Clearly $ t \neq -t $ and $ \tfrac{r^2}{t} \neq -\tfrac{r^2}{t} $ because $ t \neq 0 $ and $ \chr \left( \mathbb{F} \right) \neq 2 $. We will consider now the two cases case $ t = \pm \tfrac{r^2}{t} $. Let us start with $ t = \tfrac{r^2}{t} $, then $ t^2 = r^2 $, so $ t = \pm r $. If $ t = -\tfrac{r^2}{t} $, then $ t^2 = -r^2 $ and then $ t $ would not be admissible, so this case never occur. This means the situation can be summarized as follows: If $ t \in \boldsymbol{P} \left( \mathbb{F} \right)^{*} $ is admissible, then either all values $ \pm t, \pm \tfrac{r^2}{t} \in \boldsymbol{P} \left( \mathbb{F} \right)^{*} $ are different and we can construct two perfect distances with them or we have $ t = \pm r $ and then we can construct only one perfect distance with it by \Cref{lemma_two_points}. Hence, if all $ t \in \boldsymbol{P} \left( \mathbb{F} \right)^{*} $ are admissible, we can construct $ \tfrac{\vert \boldsymbol{P} \left( \mathbb{F} \right)^{*}\vert}{2} $ perfect distances and otherwise, i.e. if $ t \in \boldsymbol{P} \left( \mathbb{F} \right)^{*} $ with $ t^2 = -r^2 $, then $ \tfrac{\vert \boldsymbol{P} \left( \mathbb{F} \right)^{*}\vert}{2}-1 $. Observe that $ \vert \boldsymbol{P} \left( \mathbb{F} \right)^{*} \vert = \chr \left( \mathbb{F} \right) -1 $ and that the number of perfect distances we can construct gives us the number of points on a c-maximal circular point set plus one as the perfect distances can be constructed from any starting point, i.e. we have $ \tfrac{\lvert \chr \left( \mathbb{F} \right) \rvert + 1}{2} $ points in the first case and $ \frac{\lvert \chr \left( \mathbb{F} \right) \rvert - 1}{2} $ points in the second case which defines a c-maximal circular point set on a circle $ C(M,r)_{\mathbb{F}} $.

We are now ready to discuss the four cases in detail. For this we only need to decide whether $ t \in \boldsymbol{P} \left( \mathbb{F} \right)^{*} $ with $ t^2 = -r^2 $ exists or not. Observe that a $ t \in \boldsymbol{P} \left( \mathbb{F} \right)^* $ with $ t^2 = -r^2 $ only exists if and only if either $ -1,r^2 \in \square_{\boldsymbol{P} \left( \mathbb{F} \right)} $, i.e. there is a $ t \in \boldsymbol{P} \left( \mathbb{F} \right)^{*} $  with $ t^2+ 1 = 0 $ and $ r \in \boldsymbol{P} \left( \mathbb{F} \right) $ or $ -1,r^2 \notin \square_{\boldsymbol{P} \left( \mathbb{F} \right)} $ i.e. there is no $ t \in \boldsymbol{P} \left( \mathbb{F} \right)^{*} $  with $ t^2+ 1 = 0 $ and $ \pm r \notin \boldsymbol{P} \left( \mathbb{F} \right) $ by the homomorphism from the proof of \Cref{prop_fin_prim_field}. Hence, in the above cases we have $ \frac{\lvert \chr \left( \mathbb{F} \right) \rvert - 1}{2} $ points in a c-maximal circular point set and otherwise $ \frac{\lvert \chr \left( \mathbb{F} \right) \rvert +1}{2} $.

We postponed to discuss the cases where a rational squared distance $ q $ is not perfect for the cases $ 3 <  \chr \left( \mathbb{F} \right) < \infty $ (compare with \Cref{ex_non_perf}). Clearly this can only happen if there are no other perfect distances with respect to the corresponding circle because $ q $ satisfies the algebraic circle property and for other perfect distances with squared distances not equal to $ 4r^2 $ we can construct a circular point set of cardinality at least $ 3 $ by \Cref{lemma_two_points}. Hence $ q $ has the extension property by \Cref{lem_perfect_means_propagation} and is perfect. Now if there are at least three admissible values $ t \in \boldsymbol{P} \left( \mathbb{F} \right)^{*} $ then there must exist also other perfect distances than only $ 4r^2 $. Since at most two values in $ \boldsymbol{P} \left( \mathbb{F} \right)^{*} $ might not be admissible we conclude that $ q $ is always perfect if $ \vert \boldsymbol{P} \left( \mathbb{F} \right)^{*} \vert > 5 $. This means that we only have to consider the case $ \chr \left( \mathbb{F} \right) = 5 $ separately. If $ r = 1 $, then $ t = 2,3 $ is not admissible and we have $ q_{1} \left( 1 \right) = 4 = 4r^2 = q_{4} \left( 1 \right) $. Moreover, for all other choices of $ r \in \boldsymbol{P} \left( \mathbb{F} \right)^{*} $ the new circle has squared distances multiplied by $ r^2 $ with respect to the circle with radius $ r = 1 $, so the rational squared distances are invariant and we would get the same result. Hence, we see that in this case $ \frac{\lvert \chr \left( \mathbb{F} \right) \rvert - 1}{2} = 2 $ holds true. 
\newline
If $ r^2 \in \boldsymbol{P} \left( \mathbb{F} \right)^{*} $ such that $ r \notin \boldsymbol{P} \left( \mathbb{F} \right)^{*} $, then $ r^2 \in \left\{2,3 \right\} $ an so all $ t \in \boldsymbol{P} \left( \mathbb{F} \right)^{*} $ are admissible because $ t^2 + r^2 \neq 0 $. Thus, existence of a squared distances not equal to $ 4r^2 $ is satisfied and we do not have to consider this case separately.

It remains to consider the cases where $ \chr \left( \mathbb{F} \right) = \infty $ and $ r^2 \in \boldsymbol{P} \left( \mathbb{F} \right) $. Since $ r \neq 0 $, it never happens that $ r^2 + t^2 = 0 $ for $ t \in \boldsymbol{P} \left( \mathbb{F} \right) $ (note that $ \boldsymbol{P} \left( \mathbb{F} \right) $ is isomorphic to $ \mathbb{Q} $ with respect to “$ + $”  and “$ \cdot $”). If we now only consider $ \mathbb{Q}_{+} $ instead of $ \mathbb{Q} $, then it is possible to find a bijection between $ \mathbb{Q}_{+} $ and the points of a c-maximal circular point set in $ C \left(M,r \right)_{\mathbb{F}} $. For this, assign the element $ 0 $ to the fixed point. Without loss of generality, we can assume that $ r > 0 $. If $ t = r $ (note that this can only happen if $ r \in \boldsymbol{P} \left( \mathbb{F} \right) $, then we have exactly one point which has a perfect distance equal to $ 4r^2 $ and otherwise, i.e. for $ t $ and $ \tfrac{r^2}{t} $ different from each other, we find two points which has perfect distance $ q_r(t) $ to our fixed point by \Cref{prop_alg_crit}. Hence, by using the axiom of choice, we can construct a bijection from $ \mathbb{Q}_{+} $ to the points of a c-maximal circular point set which means that the cardinality of this c-maximal circular point set is equal to the cardinality of the natural numbers.
\end{proof}

Interestingly, the function $ q_r $ and a similar counting method as in the proof of \Cref{theo_general_c-maximal} can also be applied to show \Cref{fact3_finite_fields} in case the cardinality of the finite field is odd:

\begin{corollary}
	Let $ \mathbb{F} $ be a finite field such that $ \chr \left( \mathbb{F} \right) \neq 2 $. Then 
	$$ \sqrt{-1} \in \mathbb{F} \Longleftrightarrow\vert \mathbb{F} \vert \equiv 1 \pmod{4} $$ 
and
$$ \sqrt{-1} \notin \mathbb{F} \Longleftrightarrow\vert \mathbb{F} \vert \equiv 3 \pmod{4} $$ 
hold. 
\end{corollary}

\begin{proof}	
	Consider the function 
$$ q_1(t) = \left( \frac{4t}{t^2 + 1} \right)^2 $$
defined for all $ t \in \mathbb{F}^{*} $ such that $ t $ is admissible i.e. $ t^2+1 \neq 0 $. By the same arguments as in the proof of \Cref{theo_general_c-maximal} we can show that the terms $ \pm t, \pm \tfrac{1}{t} \in \mathbb{F} $ are all different from each other as far as $ t \neq \pm 1 $ for all admissible $ t $ because $ \chr \left( \mathbb{F} \right) \neq 2 $. Moreover, if  $ t \neq \pm 1 $, then $ q_1 $ evaluated on the four different elements $ t,-t, \tfrac{1}{t},-\tfrac{1}{t} $ is the same and for all other elements in $ \mathbb{F}^{*} $, $ q_1 $ will admit a value different from $ q_1 \left( t\right) $. Hence, in case $ \sqrt{-1} \in \mathbb{F} $, then $ \vert \mathbb{F}^{*} \vert-4 = \vert \mathbb{F} \vert-5 $  must be divisible by $ 4 $ and otherwise $ \vert \mathbb{F} \vert-2 = \vert \mathbb{F}^{*} \vert-3 $ must be divisible by $ 4 $.
\end{proof}
\end{subsection}
\end{section}

\begin{section}{Application in Cryptography}
\begin{subsection}{Rotation groups on circles}
	In this section we define a group on circles over arbitrary fields and we discover a connection between perfect distances and elements in the corresponding group on a circle over any prime field planes different from $ \mathbb{F}_j $ for $ j = 2,3,5 $.
	
\begin{definition} \label{def_rot_prod}
	Let $ \mathbb{F} $ be any field, $ r \in \mathbb{F}^{*} $ and $ \left( a_1,a_2 \right),\left( b_1,b_2 \right) \in C \left(\left( 0,0 \right),r \right)_{\mathbb{F}} $. Then the {\em rotation product} is defined in the following way:
	$$ \left( a_1,a_2 \right) \odot_r \left( b_1,b_2 \right) = \left( \frac{a_1b_1-a_2b_2}{r},\frac{a_1b_2+a_2b_1}{r} \right) $$
Moreover, for $ n \in \mathbb{N} $ we will inductively  define $ \left( a_1, a_2 \right)^n $ as $ n $-th power of $ \left( a_1, a_2 \right) $ with respect to the product $ \odot_r $ where $ \left( a_1, a_2 \right)^0 \coloneqq \left( r,0 \right) $.
\end{definition}
	
For simplicity, we will write $ \odot $ instead of $ \odot_r $ if it is clear from which set the considered points are. We show now that we can define a group with respect to our multiplication $ \odot $.
	
\begin{proposition}
	$ \left( C \left(\left( 0,0 \right),r \right)_{\mathbb{F}}, \odot, \left(r,0 \right) \right) $ has a group structure.
\end{proposition}

\begin{proof}
Let $ \left( a_1,a_2 \right), \left( b_1,b_2 \right), \left( c_1,c_2 \right) \in C \left(\left( 0,0 \right),r \right)_{\mathbb{F}} $. First of all, the product $ \odot $ is well-defined because
\begin{align*}
	D^2 \left( \left( \frac{a_1b_1-a_2b_2}{r},\frac{a_1b_2+a_2b_1}{r} \right), \left( 0,0 \right)\right) &= \frac{1}{r^2} \left( \left( a_1a_2-b_1b_2 \right)^2 + \left( a_1b_2+a_2b_1 \right)^2 \right) \\
	&= \frac{1}{r^2} \left( a_1^2a_2^2+b_1^2b_2^2 + a_1^2b_2^2+ a_2^2b_1^2\right) \\
	&= \frac{1}{r^2} \left(  a_1^2 + b_1^2 \right) \left( a_2^2 + b_2^2 \right)   \\
	&= r^2
\end{align*} 
and so $ \left( \frac{a_1b_1-a_2b_2}{r},\frac{a_1b_2+a_2b_1}{r} \right) \in C \left(\left( 0,0 \right),r \right)_{\mathbb{F}} $.

Moreover, it is easy to verify that $ \left( r,0 \right) \in \mathbb{F}^2 $ is the neutral element because
	$$ \left( r, 0 \right) \odot \left( a_1, a_2 \right) = \left( \frac{ra_1}{r} , \frac{ra_2}{r} \right) = \left( a_1, a_2 \right) $$
and that the element $ \left( a_1, -a_2 \right) \in C \left(\left( 0,0 \right),r \right)_{\mathbb{F}} $ is the inverse of $ \left(a_1,a_2 \right) $ since
	$$ \left( a_1,a_2 \right) \odot_r \left( a_1,-a_2 \right) = \left( \frac{a_1^2+a_2^2}{r} , \frac{-a_1a_2+ a_2a_1}{r} \right) = \left( r, 0 \right) .$$

It remains to show associativity:
	\begin{gather*}
		\big( \left( a_1,a_2 \right) \odot \left( b_1,b_2 \right) \big) \odot \left( c_1,c_2 \right) = \left( \frac{a_1b_1-a_2b_2}{r},\frac{a_1b_2+a_2b_1}{r} \right) \odot \left( c_1,c_2 \right) \\
		= \left( \frac{\left( a_1b_1-a_2b_2 \right)c_1-\left( a_1b_2+a_2b_1\right)c_2}{r^2}, \frac{\left(a_1b_1-a_2b_2 \right)c_2+\left( a_1b_2+a_2b_1\right)c_1}{r^2} \right) \\
		= \left( \frac{a_1\left(b_1c_1-b_2c_2 \right)-a_2 \left( b_1c_2+b_2c_1 \right)}{r^2}, \frac{a_1\left(b_1c_2+b_2c_1 \right)+a_2\left( b_1c_1-b_2c_2 \right)}{r^2} \right) \\
		= \left( a_1,a_2 \right) \odot \left( \frac{b_1c_1-b_2c_2}{r}, \frac{b_1c_2+b_2c_1}{r} \right) = \left( a_1,a_2 \right) \odot \big( \left( b_1,b_2 \right) \odot \left( c_1,c_2 \right) \big)
	\end{gather*}
\end{proof}		
	
In fact, the rotation product is strongly related to the complex product what we see in the next example. Observe that there are also other possible group actions on circles than the one we defined, see\cite[p.\,37-38]{Lemmermeyer2}.	

\begin{example} \label{ex_rotation}
	Let $ \mathbb{F} = \mathbb{R} $, $ r = 1 $ and $ i $ be the imaginary unit with $ i^2 = -1 $. If we interpret $ \left( a_1,b_1 \right), \left( a_2,b_2 \right) \in C \left(\left( 0,0 \right),1 \right)_{\mathbb{R}} $ as the elements $ a_1 + b_1i, a_2 + b_2i \in \mathbb{C} $, respectively, then we see that 
	$$ \left( a_1,b_1 \right) \odot \left( a_2,b_2 \right) = \left( a_1b_1-a_2b_2,a_1b_2+a_2b_1 \right) $$  
can also be interpreted as 
	$$ \left( a_1 + b_1i \right) \left( a_2 + b_2i \right) = a_1b_1-a_2b_2 + \left( a_1b_2 + a_2b_1 \right)i $$ 
which shows that we can embed $ C \left(\left( 0,0 \right),1 \right)_{\mathbb{R}} $ in $ \mathbb{C} $.
Therefore we get that the multiplication of elements in $ C \left(\left( 0,0 \right),1 \right)_{\mathbb{R}} $ is just addition of the angles if we consider them in polar coordinates. %Hence, the name rotation product seems to be  justified. Moreover, the neutral element is the one with vanishing angle in polar coordinates and the element $ a_1-a_2i $ which is the complex conjugate of $ a_1+a_2i $ is then the inverse element. In case $ r \neq 1 $, then $ \odot $ induces a rotation on the same circle. 
\end{example}

As in \Cref{def_rot_prod} we can also consider powers of points on circles. In case $ \mathbb{F} $ is finite, the subgroups generated by the elements must be cyclic. However, if we choose $ \mathbb{F} = \mathbb{Q} $, then the coordinates of the points are all rational and it is not clear whether such a subgroup is finite or not. We will treat this question starting with the following definition.
	
\begin{definition}
	We call the group $ \left( C \left(\left( 0,0 \right),r \right)_{\mathbb{F}}, \odot, \left(r,0 \right) \right) $ {\em rotation group}. Let $ \left( a_1,a_2 \right) \in C \left( \left(0,0 \right),r \right)_{\mathbb{F}} $ be a circle. Then we call the element $ \left( a_1,a_2 \right) $ {\em cyclic} or {\em acyclic} if the subgroup generated by this element is finite or infinite, respectively.
\end{definition}	
	
\begin{example} \label{ex_cycle}
	The elements $ \left( r,0 \right), \left( 0,r \right), \left(-r,0 \right), \left(0,-r \right) \in C \left( \left(0,0 \right),r \right)_{\mathbb{F}} $ are all cyclic since
	$$ \left( 0,r \right)^4 = \left( 0,-r \right)^4 = \left( -r,0 \right)^2 = \left( r,0 \right) .$$
\end{example}
	
We will show now that there are no other cyclic elements in $ C \left( \left(0,0 \right),r \right)_{\mathbb{Q}} $ than the four in \Cref{ex_cycle} for any $ r \in \mathbb{Q}^{*} $. For this we will use Gaussian integers and also some notion and notations from Chapter 1 as well as some results of it for regular and irregular primes.
	
\begin{lemma} \label{lemma_with_gauss_integers}
	Let $ x + yi \in \mathbb{Z}[i] $ where $ x,y $ are coprime and $ \N \left( x+iy\right) \in \mathbb{N} \setminus \left\{ 1 \right\} $ is a square. Then $ \left( x+yi \right)^n \notin \mathbb{N} $ for all $ n \in \mathbb{N} \setminus \left\{ 0 \right\} $.
\end{lemma}	

\begin{proof}
	Assume we have $ n \in \mathbb{N} \setminus \left\{ 0 \right\} $ such that $ \left( x+yi \right)^n \in \mathbb{N} $ and let $ k \in \mathbb{N} $ with $ \N \left( x+iy\right) = k^2 $. Then we have
	$$ k^{2n} = \N \left( x+iy\right)^n = \N \left(\left( x+iy\right)^n \right) = \left( x+iy\right)^n \overline{\left( x+iy\right)^n} =  \left( x+iy\right)^{2n} ,$$
i.e. $ \left( x+yi \right)^n = k^n $. 

Observe that $ k > 1 $, so we find a prime $ p \in \mathbb{N} $ such that $ p \mid k $ and hence $ p^n \mid k^n $, so $ p^n \mid \left( x+iy\right)^n $ in $ \mathbb{Z}[i] $. Now if $ p \in \mathbb{Z}[i] $ is regular, %(i.e. $ p \equiv 3 \pmod 4 $), 
then $ p \in \mathbb{Z}[i] $ is prime because $ \mathbb{Z}[i] $ is a unique factorization domain and so $ p \mid x + yi $ which means $ p \mid x $ and $ p \mid y $ (see Geometric Aspects to Diophantine Equations of
the Form $ x^2 + zxy + y^2 = M $ and $ z $-Rings, section 4.6, Lemma 4.39). This is a contradiction to the assumption that $ x,y $ are coprime. In case $ p \in \mathbb{Z}[i] $ is irregular, then we find $ \alpha \in \mathbb{Z}[i] $ such that $ p = \alpha \overline{\alpha} $. Now $ \alpha^n \overline{\alpha}^n = p^n \mid \left( x+iy\right)^n $ and so we have $ \alpha \mid x + yi $ and $ \overline{\alpha} \mid x + yi $. If they are not associated, then this means $ p = \alpha \overline{\alpha} \mid x + yi $ and so $ x,y $ would not be coprime as before. If $ \alpha, \overline{\alpha} $ are associated, then we find a unit $ \varepsilon \in \mathbb{Z}[i] $ such that $ \overline{\alpha} = \alpha \varepsilon $. Then we have that $ \alpha^{2n} \mid \left( x+iy\right)^n $, i.e. $ \alpha^2 \mid x + yi $ and so also $ p \mid x + yi $ which leads to the same contradiction. Thus, an $ n \in \mathbb{N} \setminus \left\{ 0 \right\} $ with $ \left( x+yi \right)^n \in \mathbb{N} $ cannot exist.
\end{proof}

\begin{corollary}
	Let $ r \in \mathbb{Q}^{*} $ be arbitrary, then all the elements in $ C \left( \left(0,0 \right),r \right)_{\mathbb{Q}} $ where both coordinates are non-vanishing are acyclic.
\end{corollary}
	
\begin{proof}
	%By the uniformity property and by translation invariance of squared distances we can assume that our circular point set is a subset of $ \subseteq C \left((0,0),1 \right)_{\mathbb{Q}} $ and contains the point $ \left( 1,0 \right) $. Observe that the points $ \left(1,0\right),\left(0,1\right),\left(-1,0\right),\left(0,-1\right) $ are the only one containing on the circle with radius $ 1 $ around the origin having a coordinate which vanishes. However, $ \left( 0, \pm 1 \right) $ has no squared rational distance to $ \left(1,0 \right) $, so they are not contained in $ C \left((0,0),1 \right)_{\mathbb{Q}} $ and $ \left( -1,0 \right) $ gives us a $ 2-cycle $ as described above. 
	Let us assume that we have a point $ \left( \tfrac{a_1}{b_1},\tfrac{a_2}{b_2} \right) \in C\left((0,0),r \right)_{\mathbb{Q}} $ such that $ a_j \neq 0 $ and $ a_j,b_j $ are coprime for $ j = 1,2 $, respectively. Let $ k_1 \in \mathbb{N} $ be the greatest common positive divisor of $ a_1,a_2 $ and $ k_2 \in \mathbb{N} $ be the least common multiple of $ b_1,b_2 $. Define $ k \coloneqq \tfrac{k_2}{k_1} $, then $ \left( x,y \right) \coloneqq \left( k\tfrac{a_1}{b_1},k\tfrac{a_2}{b_2}\right) \in \mathbb{Z}^2 $ and $ x,y $ are coprime. If we assume that $ \left( \tfrac{a_1}{b_1},\tfrac{a_2}{b_2} \right) $ is cyclic, we find $ n \in \mathbb{N} $ such that $ \left( \tfrac{a_1}{b_1},\tfrac{a_2}{b_2} \right)^n = \left( r,0 \right) $, i.e. $  \left( x,y \right)^n = \left( \left( kr \right)^n,0 \right) \in \mathbb{N}^2 $ and so $ \N \left( x + yi \right) = \left( kr \right)^{2n} \in \mathbb{N} $ is a square, but different from $ 1 $ because if $ x + yi \in \mathbb{Z}[i] $ would be a unit. Then either $ x $ or $ y $ would vanish which is not possible. If we apply \Cref{lemma_with_gauss_integers}, we get a contradiction.
\end{proof}	
	
We will see soon another property or characterization of perfect distances over prime field planes. For this we need to define a square root for rotation groups.

%We will show that the distance induced by the point of a rotation group $ C \left( M,r \right)_{F} $ for any prime field plane $ \mathbb{F} $ for $ M \in \mathbb{F}^2 $ and $ r \in \mathbb{F}^2 $ is perfect if and only if the point has a square root existing in the same group over a prime field. We start with the geometric interpretation in the case $ \mathbb{F} = \mathbb{Q} $.
	
%Similar with rotating does not exist, but we will find another explanation which also holds for other fields...	
	
\begin{definition}
	Let $ \left( a_1,a_2 \right) \in C \left( \left( 0,0 \right),r \right)_{\mathbb{F}} $, then we call the squared distance $ D^2 \left( \left( a_1,a_2 \right), \left(r,0 \right)\right) $ {\em induced squared distance by the point $ \left( a_1,a_2 \right) $ with respect to $ C \left( \left( 0,0 \right),r \right)_{\mathbb{F}} $}. We also say that a point $ \left( a_1,a_2 \right) \in C \left( \left( 0,0 \right),r \right)_{\mathbb{F}} $ has a {\em square root $ \left( b_1,b_2 \right) \in C \left( \left( 0,0 \right),r \right)_{\mathbb{F}} $} if $ \left( b_1,b_2 \right)^2 = \left( a_1,a_2 \right) $.
\end{definition}

\begin{proposition} \label{prop_perf_and_square_root}
	Let $ \mathbb{F} $ be a prime field different from $ \mathbb{F}_j $ for $ j = 2,3,5 $, $ r \in \mathbb{F}^{*} $ and $ \left(a_1,a_2 \right) \in C \left( \left( 0,0 \right),r \right)_{\mathbb{F}} $. Then $ \left(a_1,a_2 \right) $ has a square root if and only if the induced squared distance by $ \left(a_1,a_2 \right) $ is perfect.
\end{proposition}

\begin{proof}%[Proof of \Cref{prop_perf_and_square_root}]
	Let us assume that there exist $ \left( b_1,b_2 \right) \in C \left( \left( 0,0 \right),r \right)_{\mathbb{F}} $ such that 
	$$ \left( b_1,b_2 \right)^2 = \left( \frac{b_1^2-b_2^2}{r}, \frac{2b_1b_2}{r} \right) = \left( a_1,a_2 \right) .$$  
By \Cref{perf_vers_occurring} it remains to show that the induced distance of $ \left( a_1,a_2 \right) $ is a square in $ \mathbb{F} $. Observe that $ b_1^2 + b_2^2 = r^2 $. Then we have
	\begin{align*}
		D^2 \left( \left( a_1,a_2 \right), \left(r,0 \right)\right) &= \left( \frac{b_1^2-b_2^2}{r} - r \right)^2 + \frac{4b_1^2b_2^2}{r^2} \\
	&= \frac{1}{r^2} \left( \left( b_1^2-b_2^2-r^2\right)^2 + 4b_1^2b_2^2 \right) \\
	&= \frac{1}{r^2} \left( b_1^4+b_2^4+r^4-2b_1^2b_2^2 - 2b_1^2r^2+ 2b_2^2r^2+4b_1^2b_2^2 \right) \\
	&= 	\frac{1}{r^2} \left( \left( b_1^2+ b_2^2 \right)^2+ r^4-2b_1^2r^2+2b_2^2r^2 \right) \\
	&= \frac{1}{r^2} \left( \left( b_1^2+b_2^2 \right)r^2+ \left( b_1^2+b_2^2 \right)r^2- 2b_1^2r^2+ 2b_2^2r^2 \right) \\
	&= 4b_2^2 \in \square_{\mathbb{F}}.
	\end{align*}
On the other hand, assume that the squared distance $ D^2 \left( \left( a_1,a_2 \right), \left(r,0 \right)\right) $ is perfect. We would like to find $ b_1,b_2 $ such that 
$$ \left( b_1,b_2 \right)^2 = \left( \frac{b_1^2-b_2^2}{r}, \frac{2b_1b_2}{r} \right) = \left( a_1,a_2 \right) $$ 
is satisfied. Then the square roots of $ D^2 \left( \left( a_1,a_2 \right), \left(r,0 \right)\right) $ exists and so we can define $ b_2 $ as one of the square roots of
	$$ \frac{D^2 \left( \left( a_1,a_2 \right), \left(r,0 \right)\right)}{4} = \left( a_1-r \right)^2 + a_2^2 = 2r^2 -2a_1r$$ 
(note that $ a_1^2+ a_2^2 = r^2 $). We will denote one of these square roots by $ b_2 = \tfrac{\sqrt{2r^2 -2a_1r}}{2} \in \mathbb{F} $ and the choice of the sign does not matter. Now we can solve for $ b_1 $ by using the equation $ \tfrac{2b_1b_2}{r}=a_2 $, i.e. 
	$$ b_1 = \frac{r}{2b_2}a_2 = \frac{r}{\sqrt{2r^2 -2a_1r}}a_2 \in \mathbb{F}. $$
It remains to show that $ \tfrac{b_1^2-b_2^2}{r} = a_1 $ is satisfied. We have
	\begin{align*}
		\frac{b_1^2-b_2^2}{r} &= 	\frac{r}{2r^2 -2a_1r}a_2^2-\frac{2r-2a_1}{4} \\
		&= \frac{r^2-a_1^2}{2r-2a_1}-\frac{r-a_1}{2} \\
		&= \frac{r+a_1}{2}-\frac{r-a_1}{2} = a_1
	\end{align*}
which shows that $ \left( b_1,b_2 \right)^2 = \left( a_1,a_2\right) $	.
\end{proof}

\begin{example}
	We can construct perfect distances with respect to the circle $ C \left(\left( 0,0 \right),2 \right)_{\mathbb{Q}} $ in the following way: First of all, we need any rational point on $ C \left(\left( 0,0 \right),1 \right)_{\mathbb{Q}} $. To find such a point we can use Pythagorean triples. For example by $ 3^2 + 4^2 = 5^2 $ we get $ 2^2\tfrac{3^2}{5^2}+ 2^2\tfrac{4^2}{5^2} = 2^2 $, i.e. $ \left( \tfrac{8}{5}, \tfrac{6}{5}\right) \in C \left(\left( 0,0 \right),1 \right)_{\mathbb{Q}} $. The squared distance between the points $ \left( \frac{8}{5},\frac{6}{5} \right) $ and $ \left( 2,0 \right) $ is not rational because
	$$ D^2 \left( \left( \frac{8}{5},\frac{6}{5} \right), \left(2,0 \right)\right) = \frac{64+ 16}{25} = \frac{16}{5} \notin \square_{\mathbb{Q}}. $$ 
	However, if we calculate
$$ \left( \frac{8}{5},\frac{6}{5} \right) \odot \left(\frac{8}{5},\frac{6}{5} \right) = \left( \frac{64-36}{2 \cdot 25} , \frac{2 \cdot 48}{2 \cdot 25} \right) = \left( \frac{14}{25}, \frac{48}{25} \right), $$
Then the induced squared distance is
	$$ D^2 \left( \left( \frac{14}{25}, \frac{48}{25} \right), \left(2,0 \right) \right) = \frac{36^2+48^2}{625} = \left( \frac{12}{5} \right)^2 \in \square_{\mathbb{Q}} $$
as \Cref{prop_perf_and_square_root} suggests.	
\end{example}

\begin{example}
	Consider the unit circle in the Euclidean plane. We would like to construct the maximal circular point set on it containing the point $ \left( 1,0 \right) $. A parametrization of the unit circle is given by \Cref{coro_parametrisation}, i.e. each rational point in the unit circle is of the form $ \left( \frac{2t}{t^2 + 1}, \frac{ t^2 - 1 }{t^2 + 1} \right) $ for some $ t \in \mathbb{Q} $. Hence, the points
	$$ \left( \frac{2t}{t^2 + 1}, \frac{ t^2 - 1 }{t^2 + 1} \right)^2 = \left( \frac{-t^4+6t^2-1}{\left( t^2 + 1 \right)^2}, \frac{ 4t \left( t^2-1 \right) }{\left( t^2 + 1 \right)^2} \right) $$
parametrized by $ t \in \mathbb{Q} $ (or even for $ t \in \mathbb{Q}_{+} $ since $ t $ and $ -\tfrac{1}{t} $ would parametrize the same point) have rational distance to the point $ \left( 0,1 \right) $ and define a maximal circular point set by \Cref{prop_perf_and_square_root}. %That the set is maximal follows and all points have rational distance from each other follows from the definition of perfect distances, its properties and \Cref{perf_vers_occurring}.	
\end{example}
	
Note that the point set above is dense in the circle $ C \left( \left( 0,0 \right), r\right) $. That rational and dense point sets on circles in $ \mathbb{R}^2 $ exists is well-known. It is also known that there are no other irreducible algebraic curves than circles and lines that contain dense or just infinite rational point sets \cite{Solymosi_Zeeuw}. However, so far the Erd{\H o}s-Ulam problem is still open, i.e. it is not clear whether there exists an everywhere dense subset in $ \mathbb{R}^2 $ with respect to the Euclidean topology and there is still resent research on it, e.g. see  \cite{Erd_Ulam_prob_2019}.	
\end{subsection}

\begin{subsection}{A possible application in cryptography}
	The goal of this section is to describe and apply some from the previous ideas in cryptography. More concretely, we would like to use a key exchange related to the Diffie-Hellman protocol applied on circles centered at the origin instead of elliptic curves what is usually used. Recall that we defined a group on points on a circle. Now we would like to describe how two people, called A and B, can communicate with each other and agree on a point on this circle such that anyone who is in the middle gets the messages of A and B, but cannot decode it because he or she cannot solve the logarithm problem in a reasonable time.
	
For the procedure we will assume that $ \mathbb{F} $ is a prime field. We will consider both the cases when $ \mathbb{F} $ has finite and infinite characteristics. %and also discuss its advantages and disadvantages. 
At first A and B need to agree on a public key, i.e. on a point $ Q \in C \left(  \left(0,0 \right),r \right)_{\mathbb{F}} $ where $ r \in \mathbb{F} $. 
% In case $ \mathbb{F} $ is finite, then it is highly recommended to take a point which induces a squared non-rational distance because otherwise the order of the point with respect to the rotation product can only have an order which is at most equal to $ \tfrac{\vert \mathbb{F} \vert}{2} $ because otherwise the induced distance is perfect, see \Cref{perf_vers_occurring}. 
In case $ \mathbb{F} $ is infinite, then we know that all points in $ C \left(  \left(0,0 \right),r \right)_{\mathbb{F}} $ except $ \left( r,0 \right), \left( 0,r \right), \left(-r,0 \right), \left(0,-r \right)$ are acyclic. Now A can choose $ n \in \mathbb{N} $ and B chooses $ m \in \mathbb{N} $ which both are kept private. Then A calculates $ Q^n $, B calculates $ Q^m $ and send each other the calculated point. Since
	$$ \left( Q^n \right)^m =  Q^{nm} = \left( Q^m \right)^n ,$$
both A and B can calculate the key $ Q^{nm} $, but someone in the middle only knows $ Q, Q^n, Q^m $ and does not know $ Q^{nm} $ because she or he is incapable of calculating $ n,m $ which means solving the logarithm problem. Since the rotation product is equivalent to the multiplication of rotation matrices, it seems that the logarithm problem is hard to solve. The case $ \mathbb{F}= \mathbb{Q} $ might be interesting because all the elements of the group where both coordinates are not vanishing are not cyclic and the calculation of its powers are entirely in $ \mathbb{Q}^2 $, so $ n,m $ may be chosen as high as wanted to make it more secure. However, to calculate with fractions might be costly when implemented in a computer program. %so there might also appear difficulties to run it on

Whereas the calculation of $ Q^n $ can be done efficiently as 
	$$ n = \sum_{j=1}^{N}2^ja_j $$
can be decomposed binary where $ a_j \in \left\{ 0,1 \right\} $, $ N \in \mathbb{N} $ with $ a_N = 1 $. Then we can calculate
	$$ \prod_ {\substack{j=1, \\ a_j \neq 0}}^{N}Q^{2^ja_j} $$
to minimize the number of operations. Note that this so called binary exponentation can be done without calculating the binary decomposition of $ n $ explicitly.

\end{subsection} 
\end{section}

\bibliography{Maximal_rat_circ_point_sets_Chris_Busenhart}
\bibliographystyle{plain}

\end{document}